\documentclass[11pt]{amsart}
\usepackage{xcolor,graphicx}
\usepackage{fullpage}
\usepackage[OT2,T1]{fontenc}
\usepackage{mathtools}
\usepackage{hyperref}
\usepackage{pifont}
\usepackage{amsthm}
\usepackage{amsfonts}
\usepackage{amssymb}
\usepackage[mathscr]{euscript}
\usepackage[all]{xy}
\usepackage{amsmath}
\usepackage{epsfig}
\usepackage{latexsym}
\usepackage{stackengine}
\usepackage{blkarray}

\usepackage[OT2,OT1]{fontenc} \newcommand\cyr
{
\renewcommand\rmdefault{wncyr} \renewcommand\sfdefault{wncyss} \renewcommand\encodingdefault{OT2} \normalfont
\selectfont
}
\DeclareTextFontCommand{\textcyr}{\cyr} 

\newcommand*\wbar[1]{
  \hbox{ \kern-0.2em%
    \vbox{%
      \hrule height 0.5pt  
      \kern0.25ex
      \hbox{%
        \kern-0.15em
        \ensuremath{#1}%
        \kern-0.05em
      }%
    }%
  \kern0.05em}%
} 

\newcommand*\wbarnew[1]{
  \hbox{ \kern-0.2em%
    \vbox{%
      \hrule height 0.5pt  
      \kern0.25ex
      \hbox{%
        \kern-0.35em
        \ensuremath{#1}%
        \kern-0.05em
      }%
    }%
  \kern0.05em}%
}

\makeatletter
\newcommand{\bigperp}{%
  \mathop{\mathpalette\bigp@rp\relax}%
  \displaylimits
}

\newcommand{\bigp@rp}[2]{%
  \vcenter{
    \m@th\hbox{\scalebox{\ifx#1\displaystyle2.1\else1.5\fi}{$#1\perp$}}
  }%
}
\makeatother

\newtheorem{theorem}{Theorem}[subsection]
\newtheorem{lemma}[theorem]{Lemma}
\newtheorem{proposition}[theorem]{Proposition}
\newtheorem{corollary}[theorem]{Corollary}
\newtheorem{conjecture}[theorem]{Conjecture}
\newtheorem{thm}{Theorem}

\theoremstyle{definition}
\newtheorem{definition}[theorem]{Definition}
\newtheorem{example}[theorem]{Example}
\newtheorem{fact}[theorem]{Fact}
\theoremstyle{remark}
\newtheorem{remark}[theorem]{Remark}

\newcommand{\lk}{\operatorname{\ell{\it k}}}

\makeatletter
\@namedef{subjclassname@2020}{\textup{2020} Mathematics Subject Classification}
\makeatother

\title[Algebraic concordance order of almost classical knots]{Algebraic concordance order of almost classical knots}

\author[M. Chrisman]{Micah Chrisman}
\author[S. Mukherjee]{Sujoy Mukherjee}
\address{Department of Mathematics, The Ohio State University, Columbus, Ohio, 43210}
\email{chrisman.76@osu.edu}
\email{mukherjee.166@osu.edu}

\subjclass[2020]{Primary: 57K12, Secondary: 57N70}
\keywords{virtual knots, algebraic concordance group, virtual Seifert surface, almost classical knots}

\begin{document}

\begin{abstract} Torsion in the concordance group $\mathscr{C}$ of knots in $S^3$ can be studied with the algebraic concordance group $\mathscr{G}^{\mathbb{F}}$. Here $\mathbb{F}$ is a field of characteristic $\chi(\mathbb{F}) \ne 2$. The group $\mathscr{G}^{\mathbb{F}}$ was defined by J. Levine, who also obtained an algebraic classification when $\mathbb{F}=\mathbb{Q}$. While the concordance group $\mathscr{C}$ is abelian, it embeds into the non-abelian virtual knot concordance group $\mathscr{VC}$. It is unknown if $\mathscr{VC}$ admits non-classical finite torsion. Here we define the virtual algebraic concordance group $\mathscr{VG}^{\mathbb{F}}$ for almost classical knots . This is an analogue of $\mathscr{G}^{\mathbb{F}}$ for  homologically trivial knots in thickened surfaces $\Sigma \times [0,1]$, where $\Sigma$ is closed and oriented.  The main result is an algebraic classification of $\mathscr{VG}^{\mathbb{F}}$. A consequence of the classification is that $\mathscr{G}^{\mathbb{Q}}$ embeds into $\mathscr{VG}^{\mathbb{Q}}$ and $\mathscr{VG}^{\mathbb{Q}}$ contains many nontrivial finite-order elements that are not algebraically concordant to any classical Seifert matrix. For $\mathbb{F}=\mathbb{Z}/2\mathbb{Z}$, we give a generalization of the Arf invariant.
\end{abstract}
\maketitle
\setcounter{tocdepth}{1}

\section{Introduction} \label{sec_intro}

\subsection{Motivation} Two oriented knots $K_0,K_1$ in $S^3$ are said to be (smoothly) \emph{concordant} if there is a properly embedded annulus $A \subset S^3 \times [0,1]$ such that $A \cap S^3 \times 0=-K_0$ and $A \cap S^3 \times 1=K_1$. The concordance classes form an abelian group $\mathscr{C}$, with addition given by connected sum. The additive identity $0$ of $\mathscr{C}$ is equal to the concordance class of the unknot. A knot that is concordant to the unknot is called \emph{slice}. Every non-slice negative amphicheiral knot has order $2$ in $\mathscr{C}$. Other than $2$-torsion, it is unknown if there is any non-trivial finite torsion in $\mathscr{C}$. Order in the algebraic concordance group $\mathscr{G}^{\mathbb{Q}}$, however, is completely understood. Levine \cite{levine_2} proved that the rational algebraic concordance group $\mathscr{G}^{\mathbb{Q}}$ is isomorphic to $\mathbb{Z}^{\infty} \oplus (\mathbb{Z}/2\mathbb{Z})^{\infty} \oplus (\mathbb{Z}/4\mathbb{Z})^{\infty}$ and hence elements of $\mathscr{G}^{\mathbb{Q}}$ have order either $1,2,4,$ or $\infty$. Other obstructions used to determine the concordance order in $\mathscr{C}$ include the signature function, the Arf invariant, and Heegard-Floer homology (Jabuka-Naik \cite{jabuka_naik}). 
\newline
\newline
Another possible option for classifying torsion in $\mathscr{C}$ comes from the virtual knot concordance group. We briefly recall some essential definitions. Let $I=[0,1]$ and let $\Sigma_0,\Sigma_1$ be a closed oriented surfaces. Two oriented knots $K_0 \subset \Sigma_0 \times I$, $K_1 \subset \Sigma_1 \times I$ are said to be \emph{virtually concordant} if there is a compact oriented $3$-manifold $W$ and a properly embedded smooth annulus $A \subset W \times I$ such that $\partial W=\Sigma_1 \sqcup -\Sigma_0$ and $\partial A=K_1 \sqcup -K_0$ (Turaev \cite{turaev_cobordism}). In the smooth category, this is equivalent to virtual knot concordance (see Fenn-Rourke-Sanderson \cite{frs}, Kauffman \cite{lou_cob}, Carter-Kamada-Saito \cite{CKS}).  Concordance of long virtual knots is defined likewise (Chrisman \cite{band_pass}). The elements of the \emph{virtual knot concordance group} $\mathscr{VC}$ are the concordance classes of long virtual knots. The operation in $\mathscr{VC}$ is concatenation and the group identity is the long unknot. It follows from a result of Boden-Nagel \cite{boden_nagel} that the inclusion $\mathscr{C} \to \mathscr{VC}$ embeds $\mathscr{C}$ into the center of $\mathscr{VC}$.
\newline
\newline
To classify the torsion in $\mathscr{C}$, it therefore suffices to classify the torsion in $\mathscr{VC}$. For instance, if the only finite $m$-torsion in $\mathscr{VC}$ is non-classical for $m>2$, it would follow that $\mathscr{C}$ has only $2$-torsion. It is reasonable to suspect such theorems exist; the groups $\mathscr{C}$ and $\mathscr{VC}$ are known to have significantly different structure. For example, $\mathscr{VC}$ is not even abelian (Chrisman \cite{c_ext}). As a second example, recall that every classical knot is band-pass equivalent to either the trefoil or the unknot, according to whether its Arf invariant is $1$ or $0$, respectively. Since the Arf invariant is a classical concordance invariant, so also is the band-pass class. Concordant virtual knots, however, need not be band-pass equivalent. Moreover, every concordance class contains a representative that is that is not band-pass equivalent to either the trefoil or the unknot \cite{band_pass}. Some similarities between classical and virtual concordance do exist. As is the classical case, every virtual concordance class contains both a prime hyperbolic representative and a prime satellite representative (Chrisman \cite{chp}).
 \newline
 \newline
Yet, finite torsion in $\mathscr{VC}$ is elusive. It remains unknown if there are any finite order elements in $\mathscr{VC}$ that are not concordant to any classical knot. Note that if such a virtual knot exists, it must be in the kernel of any invariant $f:\mathscr{VC} \to A$ for $A$ a free abelian group. Some examples of virtual concordance invariants evaluated in free abelian groups are the odd writhe, the Henrich-Turaev polynomial, the affine index (or writhe) polynomial, and the zero polynomial. From the properties of the Artin representation \cite{c_ext}, it follows that the extended Milnor invariants of long virtual knots are also trivial on finite torsion. This implies that the generalized Alexander polynomial is also vanishing on the closure any long virtual knot of finite order. The kernels of all these invariants have a nontrivial intersection; all vanish on the set of \emph{almost classical knots}. Recall that a virtual knot is almost classical if it can be represented by a homologically trivial knot in some thickened surface. These knots, originally studied by Silver-Williams \cite{silwill0,silwill1,silwill}, are therefore fundamental to our understanding of torsion in $\mathscr{VC}$.
\newline
\newline
The above discussion suggests that we should restrict our attention to algebraic concordance of almost classical knots. Indeed, torsion in the classical algebraic concordance group is already classified and the most likely location of non-classical torsion in $\mathscr{VC}$ lies in set of almost classical knots. Furthermore, as each almost classical knot has a homologically trivial representative in some $\Sigma \times I$, each bounds a Seifert surface in $\Sigma \times I$. Seifert matrices can be defined in the usual way and algebraic concordance can be investigated by analogy with the classical case. The main theme of the present paper is the algebraic concordance order of such Seifert matrices. The remainder of this section outlines our results in this direction.

\subsection{Section 2: Algebraic concordance} Let $\mathbb{F}$ be a field of characteristic $\chi(\mathbb{F}) \ne 2$. After reviewing the Seifert forms of almost classical knots in Section \ref{sec_review_sf}, we define the virtual algebraic concordance group $\mathscr{VG}^{\mathbb{F}}$ in Section \ref{sec_uncoupled}. The following classification of $\mathscr{VG}^{\mathbb{F}}$ is the main result of the paper. Recall that if $\mathscr{W}(\mathbb{F})$ is the Witt ring of regular quadratic forms over $\mathbb{F}$, then its fundamental ideal $\mathscr{W}(\mathbb{F})$ consists of the even rank forms in $\mathscr{W}(\mathbb{F})$.

\begin{thm} \label{thm_classify_VG} $\mathscr{VG}^{\mathbb{F}} \cong \mathscr{I}(\mathbb{F}) \oplus \mathscr{G}^{\mathbb{F}}$, where $\mathscr{G}^{\mathbb{F}}$ is the classical algebraic concordance group over $\mathbb{F}$ and $\mathscr{I}(\mathbb{F})$ is the fundamental ideal of the Witt ring $\mathscr{W}(\mathbb{F})$.
\end{thm}

An immediate corollary is that $\mathscr{VG}^{\mathbb{Q}}$ has torsion only of order 1,2, 4, and $\infty$. Theorem \ref{thm_classify_VG} suggests a natural question: when are the Seifert matrices of an almost classical knot algebraically concordant to those of a classical knot? To study this, we introduce an extension of $\mathscr{VG}^{\mathbb{F}}$ called the \emph{coupled algebraic concordance group} $(\mathscr{VG},\mathscr{VG})^{\mathbb{F}}$. Its elements are represented by pairs $(A^+,A^-)$ of Seifert matrices, where $A^{\pm}$ correspond to the $\pm$-push-offs in $\Sigma \times I$. The pair $\textbf{A}=(A^+,A^-)$ is called a \emph{Seifert couple}. In Section \ref{sec_coupled}, we show that if $A^+$ and $A^-$ are not concordant in $\mathscr{VG}^{\mathbb{Q}}$, then $(A^+,A^-)$ is not concordant in $(\mathscr{VG},\mathscr{VG})^{\mathbb{Q}}$ to a pair of classical Seifert matrices.
\newline
\newline
For $\mathbb{F}=\mathbb{Z}/2\mathbb{Z}$, the Seifert matrices $A^{\pm}$ define a unique quadratic form $q_{\textbf{A}}$. In Section \ref{sec_arf} , we show that the Arf invariant of a regular quadratic form obstructs algebraic sliceness. In particular, if $\textbf{A}$ is algebraically slice and $q_{\textbf{A}}$ is regular, then $\text{Arf}(q_{\textbf{A}})=0$.


\subsection{Section 3: Geometric concordance} Having completed the algebraic classification, we proceed in Section \ref{sec_gc} to give a geometric interpretation of the groups $\mathscr{VG}^{\mathbb{Z}}$ and $(\mathscr{VG},\mathscr{VG})^{\mathbb{Z}}$ with coefficients in $\mathbb{Z}$. Two Seifert surfaces $F_0 \subset \Sigma_0 \times I$, $F_1 \subset \Sigma_1 \times I$ will be called \emph{virtually concordant} if the knots $K_0=\partial F_0$, $K_1 =\partial F_1$ are virtually concordant via an annulus $A \subset W \times I$ and $-F_0 \cup A \cup F_1$ bounds a compact oriented $3$-manifold in $W \times I$. This can be viewed as a generalization of Myers' notion of concordant Seifert surfaces \cite{myers_new}.  Explicit examples of virtually concordant Seifert surfaces are given in Section \ref{sec_virt_conc_defn}. The main theorem relating geometric and algebraic concordance is the following.

\begin{thm} \label{thm_B} If $F_0\subset \Sigma_0 \times I$,$F_1 \subset \Sigma_1 \times I$ are virtually concordant Seifert surfaces, then their Seifert couples $(A_0^+,A_0^-)$, $(A_1^+,A_1^-)$ are algebraically concordant in $(\mathscr{VG},\mathscr{VG})^{\mathbb{Z}}$.
\end{thm}

A fixed knot $K_0\subset \Sigma_0 \times I$ may have many Seifert surfaces in $\Sigma_0 \times I$ that are not virtually concordant to each other. However, all virtual concordance classes can be generated from any single representative $F_0 \subset \Sigma \times I$. In Section \ref{sec_ambient_Z}, we show that if a Seifert surface $F_1 \subset \Sigma_1 \times I$ of $K_1$ is virtually concordant to $F_0$ and $F_1' \subset \Sigma \times I$ is another Seifert surface for $K_1$, then there is a Seifert surface $F_0' \subset \Sigma_0 \times I$ of $K_0$ that is virtually concordant to $F_1'$ and can be obtained from $F_0$ by a sequence of isotopies, $S$-equivalences, and connected sums with a parallel copies of $\Sigma \times 1$. 

\subsection{Section 4: Calculations $\&$ Examples} The theory of Sections \ref{sec_ac} and \ref{sec_gc} is illustrated with examples in Section \ref{sec_examples}. The examples can be most easily constructed using virtual Seifert surfaces (Chrisman \cite{vss}). These are reviewed in Section \ref{sec_vss}. A sample calculation of the Arf invariant is given in Section \ref{sec_arf_example}. A guide to calculating algebraic concordance order in $\mathscr{VG}^{\mathbb{Q}}$ appears in Section \ref{sec_calculating}. In Section \ref{sec_nonclass_torsion}, we show that there is an almost classical knot having Seifert couple $\textbf{A}=(A^+,A^-)$ such that $A^+$ has order $2$ and $A^-$ has order $1$. Hence $\textbf{A}$ is not concordant in $(\mathscr{VG},\mathscr{VG})^{\mathbb{Q}}$ to any pair of classical Seifert matrices. The knot is itself not concordant to any classical knot. This gives some positive evidence to the conjecture that $\mathscr{VC}$ admits non-classical finite torsion.
\newline
\newline
Given the previous example, an obvious question is whether there exists an almost classical knot having all possible finite orders in algebraic concordance. Infinitely many examples of such knots exist, as is proved in the following theorem (see Section \ref{sec_thm_C_proof}).

\begin{thm} \label{thm_C} There exist infinitely many knots $K$ in $S^1 \times S^1 \times I$ such that for each $o \in \{1,2,4\}$, $K$ bounds a Seifert surface $F \subset S^1 \times S^1 \times I$ having algebraic concordance order $o$ in $\mathscr{VG}^{\mathbb{Q}}$.
\end{thm}

We have endeavored to provide an account that is both practical and (mostly) self-contained.  It is hoped that readers new to algebraic concordance will afterwards feel confident enough to perform their own calculations. For more detailed information on the classical algebraic concordance group, we recommend the works of Collins \cite{collins}, Levine \cite{levine_2,levine}, Livingston \cite{liv_aco}, Livingston-Naik \cite{livingston_naik}, and Milnor \cite{milnor_isom}. 

\section{Algebraic concordance} \label{sec_ac}

\subsection{Directed Seifert forms} \label{sec_review_sf} Here we review the directed Seifert forms of almost classical knots from \cite{bcg2}. Let $\Sigma$ be a closed oriented surface. We first recall the definition of the linking number in $\Sigma \times I$. For a knot $J \subset \Sigma \times I$, we have $H_1(\Sigma \times I \smallsetminus J,\Sigma \times 1;\mathbb{Z}) \cong \mathbb{Z} \cong \langle \mu \rangle$, where $\mu$ is a meridian of $J$. If $J\sqcup K \subset \Sigma \times I$ is a link, then $[K]=m \cdot [\mu] \in H_1(\Sigma \times I\smallsetminus J,\Sigma \times 1;\mathbb{Z})$ for some $m \in \mathbb{Z}$. Then define $\lk_{\Sigma}(J,K)=m$.  In general, $\lk_{\Sigma}(J,K) \ne \lk_{\Sigma}(K,J)$. By Cimasoni-Turaev \cite{ct}, the relationship is instead given by:
\begin{align} \label{eqn_ct}
\lk_{\Sigma}(J,K)-\lk_{\Sigma}(K,J) &=p_*([J])\bullet_{\Sigma}p_*([K]).
\end{align}
Throughout the paper, we will write $-\bullet_X-$ for the intersection form on $X$ (see Dold \cite{dold}, VIII.13). Given a diagram of a link $J \sqcup K$ in $\Sigma\times I$, $\lk_{\Sigma}(J,K)$ is the signed sum of over-crossings of $J$ with $K$.  
\newline
\newline
Let $K \subset \Sigma \times I$ be a homologically trivial knot. By a \emph{Seifert surface} for $K$, we mean any connected, compact, oriented surface $F \subset \Sigma \times I$ such that $\partial F=K$. For a $1$-cycle $x$ on $F$, let $x^{\pm}$ denote that $\pm$-push-off of $x$ into $\Sigma \times I \smallsetminus F$. The adjective \emph{directed} refers to a fixed choice of one of these two push-offs. The \emph{directed Seifert forms} $\theta^{\pm}_{K,F}:H_1(F;\mathbb{Z})\times H_1(F;\mathbb{Z}) \to \mathbb{Z}$ are given by:
\[
\theta^{\pm}_{K,F}(x,y)=\lk_{\Sigma}(x^{\pm},y).
\]
For a basis $\xi_F=\{\alpha_1,\ldots,\alpha_{2g}\}$ of $H_1(F;\mathbb{Z})$, the \emph{directed Seifert matrices} are $A^{\pm}=(\lk_{\Sigma}(\alpha_i^{\pm},\alpha_j))$. As in the case of classical knots in $S^3$, $A^--A^+$ is a matrix of the intersection form on $F$ (see e.g Burde-Zieschang \cite{bz}, Proposition 8.7). Contrary to the case of classical knots, it is not in general true that $A^+=(A^-)^{\intercal}$. Hence there are two different, but related, Seifert forms for any given Seifert surface. The most important properties of directed Seifert matrices are recorded below.

\begin{lemma} \label{lemma_seif_mat_facts} Let $A^{\pm}$ be directed Seifert matrices of a Seifert surface $F \subset \Sigma \times I$. Then:
\begin{enumerate}
\item $A^--A^+$ is skew-symmetric,
\item $\det(A^--A^+)=1$,
\item $A^++(A^+)^{\intercal}=A^-+(A^-)^{\intercal}$.
\end{enumerate}
\end{lemma}
\begin{proof} For  items (1) and (2), see \cite{acpaper}, Lemma 7.5. Item (3) follows directly from item (1). See also \cite{bcg2}, Lemma 2.1.
\end{proof}

As will be seen ahead, it is fruitful to study directed Seifert matrices both independently and as a pair $(A^+,A^-)$. As independent matrices $A^{\pm}$, each is controlled by a separate Alexander polynomial. 

\begin{definition}[Directed Alexander polynomials \cite{bcg2}] For directed Seifert matrices $A^{\pm}$, The \emph{directed Alexander polynomials} are defined by $\Delta_{K,F}^{\pm}(t)=\det(A^{\pm}-t \cdot (A^{\pm})^{\intercal})$.
\end{definition}

\begin{remark} \label{remark_alex_roots} Unlike the classical case, directed Alexander polynomials can have both $1$ and $-1$ as roots. This commonly occurs, as will be  see the examples in Section \ref{sec_examples} below. Here we will focus exclusively on the case that $-1$ is not a root, so that the theory of symmetric bilinear forms may be readily applied. The case that $1$ is not a root requires skew-symmetric bilinear forms. This is closely related to Turaev's theory of graded matrices \cite{turaev_cobordism} and will be considered by the authors in a separate paper. See Section \ref{sec_further}.
\end{remark}

\subsection{The virtual algebraic concordance group} \label{sec_uncoupled} Before defining $\mathscr{VG}^{\mathbb{F}}$, we recall the definition of the classical algebraic concordance group $\mathscr{G}^{\mathbb{F}}$, for $\mathbb{F}$ a field of characteristic $\chi(\mathbb{F}) \ne 2$. Let $A$ be a $2n \times 2n$ dimensional matrix. Then $A$ is said to be \emph{metabolic} (or \emph{algebraically slice}) if there is a matrix $P$ with entries in $\mathbb{F}$ such that $\det(P) \ne 0$ and $PAP^{\intercal}$ has block form:
\[
\left[\begin{array}{c|c} 0 & B \\ \hline C & D \end{array}\right],
\]  
where $0$ is the $n \times n$ matrix of zeros. Two even dimensional matrices $A_0,A_1$ are said to be \emph{concordant} if $A_1 \oplus -A_0$ is metabolic. If $A_0,A_1$ are concordant, we will write $A_0 \sim A_1$. To show this defines a group structure, Levine proved the following Witt cancellation property (see \cite{levine}, Lemma 1). Its proof is needed ahead in our proof of Lemma \ref{lemma_witt_cancel}. We give a somewhat different proof of Witt cancellation for the reader's convenience. 

\begin{lemma}\label{lemma_witt_cancel_0} Let $N$, $A$ be matrices over $\mathbb{F}$ of dimension $2k,2m$, respectively. Suppose that $N$ and $A\oplus N$ are metabolic. If $\det(N+N^{\intercal}) \ne 0$, then $A$ is metabolic. 
\end{lemma}
\begin{proof} For $W=\mathbb{F}^d$ a vector space over $\mathbb{F}$ and $v \in W$, let $v^* \in \text{Hom}(W,\mathbb{F})$ be the dual map $v^*(w)=v^{\intercal}w$. View $A$, $N$ as linear maps $\mathbb{F}^{2m} \to \mathbb{F}^{2m}$, $\mathbb{F}^{2k} \to \mathbb{F}^{2k}$. By hypothesis, we may write $\mathbb{F}^{2k}=\mathbb{F}^k\oplus \mathbb{F}^l$, where $k=l$ and  $(v \oplus 0)^*N$ vanishes on $\mathbb{F}^k$ for all $v \in \mathbb{F}^k$. Then the following diagram commutes for all $x \in \mathbb{F}^{2m}$, $y \in \mathbb{F}^k$, and $z \in \mathbb{F}^l$:
\[
\xymatrix{\mathbb{F}^k \oplus \mathbb{F}^l \ar[r]\ar[d]_N & \mathbb{F}^{2m} \oplus \mathbb{F}^k \oplus \mathbb{F}^l \ar[d]_{A \oplus N} & \ar[l] \ar[d]_A \mathbb{F}^{2m} \\
\mathbb{F}^k \oplus \mathbb{F}^l \ar[r] \ar[d]_{y^* \oplus z^*} & \mathbb{F}^{2m} \oplus \mathbb{F}^k \oplus \mathbb{F}^l \ar[d]_{x^*\oplus y^* \oplus z^*} & \ar[l] \mathbb{F}^{2m} \ar[d]_{x^*} \\
\mathbb{F} \ar@{=}[r] & \mathbb{F} & \ar@{=}[l]\mathbb{F}
}
\]
The horizontal arrows above are the canonical inclusions. Let $V \subset \mathbb{F}^{2m}\oplus \mathbb{F}^k \oplus \mathbb{F}^l$ satisfy $\dim(V)=m+k$ and $v^* (A\oplus N)(V)=0$ for all $v \in V$. Choose a basis for $V$ of the form:
\[
\{x_1\oplus y_1 \oplus z_1, \ldots,x_r\oplus y_r \oplus z_r,x_{r+1} \oplus y_{r+1} \oplus 0, \ldots, x_{r+s}\oplus y_{r+s} \oplus 0, 0 \oplus y_{r+s+1} \oplus 0,\ldots ,0 \oplus y_{m+k} \oplus 0 \}, 
\]
with the sets $\{z_1,\ldots,z_r\}$, $\{x_{r+1},\ldots,x_{r+s}\}$, and $\{y_{r+s+1},\ldots,y_{m+k}\}$ all linearly independent. Applying the middle vertical composition in the diagram to the elements $x_{r+1} \oplus y_{r+1} \oplus 0,\ldots, x_{r+s} \oplus y_{r+s} \oplus 0$, it follows that $x_i^* A(x_j)=0$ (since $(y\oplus 0)^*N(w\oplus 0)=0$ for $y,w \in \mathbb{F}^k$). We must show $s \ge m$. 
\newline
\newline
Again using the diagram, we obtain $(y_i \oplus 0)^* N (0\oplus z_j)=0$ and $(0 \oplus z_j)^* N (y_i\oplus 0)=0$ for $i \ge r+s+1$.  Then $(y_i \oplus 0)^* (N+N^{\intercal}) (0\oplus z_j)=0$. Hence, $(N+N^{\intercal})(y_i \oplus 0) \in \langle z_1 ,\ldots,z_r \rangle^{\perp} \cap \mathbb{F}^l$ for $r+s+1 \le i \le m+k$. Since $\det(N+N^{\intercal}) \ne 0$,  we have $m+k-(r+s) \le k-r$. Thus, $s \ge m$.
\end{proof}

From this Witt cancellation property, it follows that $\sim$ defines an equivalence relation on the set of even dimensional matrices $A$ satisfying $\det(A+A^{\intercal}) \ne 0$. Furthermore, block sum $\oplus$ defines commutative and associative operation. For more details, we refer the reader to Levine \cite{levine}, Section 3, or Livingston-Naik \cite{livingston_naik}, Theorem 4.5.4. If $\det((A-A^{\intercal})(A+A^{\intercal})) \ne 0$, then $A$ is called an \emph{$\mathbb{F}$-Seifert matrix}. From Lemma \ref{lemma_witt_cancel_0}. Setting $0$ to be the concordance class of metabolic matrices, the $\sim$-equivalence classes of $\mathbb{F}$-Seifert matrices then form a group, where the additive inverse of $A$ is $-A$. This group is called the \emph{algebraic concordance group} over $\mathbb{F}$, denoted $\mathscr{G}^{\mathbb{F}}$.
\newline
\newline
Now consider the directed Seifert matrices from Section \ref{sec_review_sf}. Our goal is to apply the theory of regular quadratic forms to symmetrized directed Seifert matrices $A+A^{\intercal}$. By Remark \ref{remark_alex_roots}, $\det(A-A^{\intercal})$ can be $0$. Hence, we define a  \emph{$\mathbb{F}$-directed (Seifert) matrix} to be a $2n \times 2n$ dimensional matrix over $\mathbb{F}$ satisfying only $\det(A+A^{\intercal}) \ne 0$. Again by Lemma \ref{lemma_witt_cancel_0}, the concordances classes of directed matrices over $\mathbb{F}$ form a group with operation $\oplus$ and additive identity $0$.

\begin{definition}[Virtual algebraic concordance group] \label{defn_uncoupled} The virtual algebraic concordance group, denoted $\mathscr{VG}^{\mathbb{F}}$, is the set $\sim$-equivalence classes of $\mathbb{F}$-directed matrices. 
\end{definition}

\begin{example} \label{example_mnmat_1} The following directed matrices will serve as a running example throughout the text. Let $m,n \ge 1$ be natural numbers. Define:
\[
A^+_0=\begin{bmatrix} m & 0 \\ 0 & -n \end{bmatrix}, \quad \quad A^-_0=\begin{bmatrix} m & 1 \\ -1 & -n \end{bmatrix}
\]
Note that $A_0^{\pm}$ satisfy the conclusion of Lemma \ref{lemma_seif_mat_facts}, since $A_0^--A_0^+$ is skew-symmetric, $\det(A_0^--A_0^+)=1$, and $A_0^++(A_0^+)^{\intercal}=A_0^-+(A_0^-)^{\intercal}$. Furthermore, we have that $\det(A_0^++(A_0^+)^{\intercal})=-4mn \ne 0$ and $\det(A_0^+-(A_0^+)^{\intercal})=0$. Thus, $A_0^+$ represents an element of $\mathscr{VG}^{\mathbb{Q}}$ that is not in $\mathscr{G}^{\mathbb{Q}}$. On the other hand, $\det(A_0^--(A_0^-)^{\intercal})=4$. Hence, $A_0^-$ represents an element of both $\mathscr{VG}^{\mathbb{Q}}$ and $\mathscr{G}^{\mathbb{Q}}$. \hfill $\square$
\end{example}

\subsection{Directed isometric structures} Efficient calculation in $\mathscr{G}^{\mathbb{F}}$ is made possible using the theory of isometric structures (Milnor \cite{milnor_isom}). This is also the case for $\mathscr{VG}^{\mathbb{F}}$, as we explain in this section. Recall that an \emph{isometric structure} is triple $(V,B,S)$ where $V$ is an even dimensional vector space over $\mathbb{F}$, $B: V \times V \to \mathbb{F}$ is a regular symmetric bilinear form on $V$ and $S$ is an isometry of $B$. A triple is called \emph{metabolic} if $B$ has an $S$-invariant metabolic subspace $W \subset V$. Isometric structures $(V_0,B_0,S_0), (V_1,B_1,S_1)$ are \emph{concordant} if $(V_1 \oplus V_0,B_1 \oplus -B_0,S_1 \oplus S_0)$ is metabolic. Let $\Delta_S(t)=\det(S-tI)$ denote the characteristic polynomial of $S$. An isometric structure $(V,B,S)$ is called \emph{admissible} if $\Delta_S(1)\Delta_S(-1) \ne 0$. Denote by $\mathscr{G}_{\mathbb{F}}$ the set of concordance classes of admissible isometric structures. For the case of directed matrices, we must allow for the case that $1$ is a root of the directed Alexander polynomial. This motivates the following definition.

\begin{definition}[Directed isometric structure] \label{defn_directed} A \emph{directed isometric structure} is an isometric structure $(V,B,S)$ as above such that $\Delta_S(-1) \ne 0$. The set of concordance classes with operation $\oplus$ forms an abelian group $\mathscr{VG}_{\mathbb{F}}$ with operation $\oplus$ and additive identity the concordance class of metabolic directed isometric structures. 
\end{definition}

In \cite{levine}, Levine proved that there is an isomorphism $ \eta_{\mathbb{F}}: \mathscr{G}^{\mathbb{F}} \to \mathscr{G}_{\mathbb{F}}$. Thus, one can work directly with symmetric bilinear forms and their isometries. The map $\eta_{\mathbb{F}}$ extends to $\mathscr{VG}^{\mathbb{F}}$ and gives an isomorphism with $\mathscr{VG}_{\mathbb{F}}$. The argument is sketched below as the map $\eta_{\mathbb{F}}$ is needed in our calculations.

\begin{theorem} \label{thm_eta_iso} The map $\eta_{\mathbb{F}}:\mathscr{VG}^{\mathbb{F}} \to \mathscr{VG}_{\mathbb{F}}$ defined by $\eta_{\mathbb{F}}(A)=(\mathbb{F}^{2n},A+A^{\intercal},A^{-1}A^{\intercal})$, where $A$ is non-singular directed matrix of rank $2n$, is an isomorphism.
\end{theorem}

\begin{proof}[Sketch of proof] It needs only be shown that the weakened hypotheses of Definition \ref{defn_directed} suffice to prove that $\eta_{\mathbb{F}}$ is an isomorphism. First recall that every singular $\mathbb{F}$-Seifert matrix is concordant to a non-singular $\mathbb{F}$-Seifert matrix or the empty matrix (see \cite{levine}, Lemma 8, or \cite{livingston_naik}, Theorem 4.72). The algorithm for finding a non-singular representative also works for $\mathbb{F}$-directed matrices without alteration. Thus it may be assumed that $\eta_{\mathbb{F}}$ is defined on non-singular representatives of each concordance class.
\newline
\newline
That $A+A^{\intercal}$ defines a symmetric bilinear form is clear. To show that $A^{-1}A^{\intercal}$ is an isometry, observe that:
\[
(A^{-1}A^{\intercal})^{\intercal}(A+A^{\intercal})A^{-1}A^{\intercal}= A (A^{\intercal})^{-1}AA^{-1}A^{\intercal}+A (A^{\intercal})^{-1}A^{\intercal}A^{-1}A^{\intercal}=A+A^{\intercal}
\] 
Since $\Delta_{A^{-1}A^{\intercal}}(-1)=\det(A^{-1}A^{\intercal}+I)=\det(A^{-1})\det(A+A^{\intercal}) \ne 0$, it follows that $\eta_{\mathbb{F}}(A) \in \mathscr{VG}_{\mathbb{F}}$.
\newline
\newline
Given a directed isometric structure $(\mathbb{F}^{2n},B,S)$, set $A=B(I+S)^{-1}$. Since $\det(S+I)=\Delta_S(-1) \ne 0$, we have that $S+I$ is non-singular. This implies that  $(S^{-1}+I)=S^{-1}(I+S)$ is non-singular. Since $B$ is non-singular, $A$ is well-defined and non-singular. Now note that $A^{\intercal}=B(I+S^{-1})^{-1}$. This follows from the fact the following calculation and the fact that $S$ is an isometry:
\[
I=(B^{-1}+B^{-1}S^{\intercal})A^{\intercal}=(B^{-1}+S^{-1}B^{-1})A^{\intercal}=(I+S^{-1})B^{-1}A^{\intercal}. 
\]
Then we have:
\begin{eqnarray*}
A^{-1}A^{\intercal} &=& (I+S)B^{-1}B(I+S^{-1})^{-1}=(I+S)(S+I)^{-1}(S^{-1})^{-1}=S \\
A+A^{\intercal}&=& B(I+S)^{-1}+B(I+S^{-1})^{-1}=B((I+S)^{-1}+(I+S^{-1})^{-1})=BI=B.
\end{eqnarray*}
Hence, $A \in \mathscr{VG}^{\mathbb{F}}$ and $\eta_{\mathbb{F}}(A)=(V,B,S)$. It is easily seen that $\eta_{\mathbb{F}}$ is a well-defined homomorphism on concordance classes. From the above argument, it follows that $\eta_{\mathbb{F}}$ is surjective and injective.
\end{proof}

\begin{example} Consider again the $\mathbb{Q}$-directed matrices $A_0^{\pm}$ from Example \ref{example_mnmat_1} with $mn>1$. Set $B_0=A_0^{\pm}+(A_0^{\pm})^{\intercal}$ and $S^{\pm}=(A_0^{\pm})^{-1})(A_0^{\pm})^{\intercal}$. Then:
\[
B_0=\begin{bmatrix} 2m & 0 \\ 0 & -2n \end{bmatrix}, \quad S^+_0=\begin{bmatrix} 1 & 0 \\ 0 & 1 \end{bmatrix}, \quad S^-_0=\frac{1}{1-mn}\begin{bmatrix} -1-mn & 2n \\ 2m & -1-mn \end{bmatrix}
\]
Note that these give two different isometric structures $(\mathbb{Q}^2,B_0,S^+_0)$ and $(\mathbb{Q}^2,B_0,S^-_0)$ for the same symmetric rational bilinear form $B_0$. For the case $m=n=1$, $A^{\pm}_0$ are both metabolic. \hfill $\square$
\end{example}

\subsection{Proof Theorem \ref{thm_classify_VG}}  We first recall some well-known facts about decomposing general isometric structures into primary components (see Levine \cite{levine_2}, Milnor \cite{milnor_isom}). Here we follow \cite{livingston_naik}, Section 4.9. If $\lambda(t) \in \mathbb{F}[t]$, define $\bar{\lambda}(t)=t^d \lambda(t^{-1})$ where $d \in \mathbb{N}$ is chosen large enough so that $\bar{\lambda}(t) \in \mathbb{F}[t]$ but small enough so that its constant term is nonzero. An irreducible polynomial $\lambda(t) \in \mathbb{F}[t]$ over $\mathbb{F}$ is called \emph{non-symmetric} if $\lambda(t)$ and $\bar{\lambda}(t)$ have greatest common divisor $1$. An irreducible polynomial over $\mathbb{F}$ that is not non-symmetric is called \emph{symmetric}. For an arbitrary isometric structure (not necessarily admissible), it follows from \cite{levine_2}, Lemma 7a, that $\Delta_S(t)=a t^{2n}\Delta_S(t^{-1})$ for some $a \in \mathbb{F}$ and $2n=\dim(V)$. Thus, if $\lambda(t)$ is a non-symmetric irreducible factor of $\Delta_S(t)$, then so is $\bar{\lambda}(t)$.
\newline
\newline
Let $(V,B,S)$ be an isometric structure. For an irreducible factor $\lambda(t)$ of $\Delta_S(t)$, denote the $\lambda(t)$-primary component of $V$ by $V_{\lambda(t)}$. Here we are viewing $V$ and a module over $\mathbb{F}[t]$ (or $\mathbb{F}[t,t^{-1}]$) with action $t \cdot v=S(v)$. Then we may identify the $\lambda(t)$-primary component of $V$ with $\text{ker}\left( \lambda(t)^N\right)$ for $N$ sufficiently large. By \cite{levine_2}, Lemma 10, if $\lambda_1(t),\lambda_2(t)$ are irreducible factors of $\Delta_S(t)$ and $\lambda_1(t),\bar{\lambda}_2(t)$ are relatively prime, then $V_{\lambda_1(t)},V_{\lambda_2(t)}$ are orthogonal relative to $B$. Given the prime factorization of $\Delta_S(t)=\prod_i (\lambda_i(t))^{k_i}$ it follows that $V=\bigoplus_i V_{\lambda_i(t)}$. Note that if $\lambda(t)$ is non-symmetric, $\lambda(t),\bar{\lambda}(t)$ are relatively prime and hence $V$ has a direct summand of the form $V_{\lambda(t)} \oplus V_{\bar{\lambda}(t)}$. 
\newline
\newline
An isometric structure $(V,B,S)$ may therefore be decomposed as $(\bigoplus_i V_{\lambda_i(t)},\bigoplus_i B_{\lambda_i(t)},\bigoplus_i S_{\lambda_i(t)})$, where $B_{\lambda_i(t)}$ and $S_{\lambda_i(t)}$ are the restrictions of $B$ and $S$, respectively, to $V_{\lambda_i(t)}$.  We now apply these observations to $\mathscr{VG}_{\mathbb{F}}$. For $\lambda(t)$ irreducible, denote by $\mathscr{VG}_{\lambda(t)} \subset \mathscr{VG}_{\mathbb{F}}$ the subgroup generated by $\mathbb{F}$-directed isometric structures $(V,B,S)$ such that $\Delta_S(t)=(\lambda(t))^k$ for some $k$. Similarly one may define a subgroup $\mathscr{G}_{\lambda(t)}$ of $\mathscr{G}_{\mathbb{F}}$. 

\begin{lemma}\label{lemma_primary_decomposition} The classical and virtual algebraic concordance group have decompositions:
\[
\mathscr{G}_{\mathbb{F}} \cong \bigoplus_{\lambda(t) \ne t \pm 1} \mathscr{G}_{\lambda(t)}, \quad \mathscr{VG}_{\mathbb{F}} \cong \bigoplus_{\lambda(t) \ne t+1} \mathscr{VG}_{\lambda(t)}.
\]
The sums range over symmetric irreducible polynomials $\lambda(t) \in \mathbb{F}[t]$ with the stated restrictions.
\end{lemma}
\begin{proof} The first isomorphism was proved in Levine \cite{levine_2}, Lemma 11. Consider the decomposition of $\mathscr{VG}_{\mathbb{F}}$. For any directed isometric structure $(V,B,S)$, $t+1$ does not divide $\Delta_S(t)$ since $\Delta_S(-1) \ne 0$ by assumption. Thus, $\mathscr{VG}_{t+1}$ is empty. On the other hand, the subgroup $\mathscr{VG}_{\lambda(t)} \oplus \mathscr{VG}_{\bar{\lambda}(t)} \subset \mathscr{VG}_{\mathbb{F}}$ is trivial when $\lambda(t)$ is non-symmetric. Indeed, $V_{\lambda(t)}$ is a metabolic subspace of $V_{\lambda(t)} \oplus V_{\bar{\lambda}(t)}$ for $B_{\lambda(t)} \oplus B_{\bar{\lambda}(t)}$. Thus, $\mathscr{VG}_{\mathbb{F}}$ is the orthogonal sum of $\mathscr{VG}_{\lambda(t)}$ where $\lambda(t) \ne t+1$ is symmetric and irreducible over $\mathbb{F}$. This also gives the first isomorphism, since $\Delta_S(1) \ne 0$ for $(V,B,S) \in \mathscr{G}_{\mathbb{F}}$. 
\end{proof}

Denote by $\mathscr{W}(\mathbb{F})$ the Witt group of regular symmetric bilinear forms over $\mathbb{F}$ and let $\mathscr{I}(\mathbb{F}) \subset \mathscr{W}(\mathbb{F})$ be the subgroup consisting of those forms with even rank. In other words, $\mathscr{I}(\mathbb{F})$ is the fundamental ideal of the Witt ring $\mathscr{W}(\mathbb{F})$, considered as a group.

\begin{lemma}\label{lemma_t_minus_1_easy} The group $\mathscr{VG}_{t-1} \subset \mathscr{VG}_{\mathbb{F}}$ satisfies $\mathscr{VG}_{t-1} \cong \mathscr{I}(\mathbb{F})$.
\end{lemma}
\begin{proof} Let $B: V \times V \to \mathbb{F}$ be a regular symmetric bilinear form of rank $2n$. The identity map $I:V \to V$ is an isometry of $B$ and hence $(V,B,I)$ is an isometric structure. The characteristic polynomial of $I$ is $\Delta_I(t)=(t-1)^{2n}$. Then $\Delta_I(-1)=2^{2n} \ne 0$ and $\Delta_I(0)=1$. Thus, $(V,B,I)$ is an $\mathbb{F}$-directed isometric structure and there is a function $\mathscr{I}(\mathbb{F}) \to \mathscr{VG}_{t-1}$ defined by $(V,B) \to (V,B,I)$. If $(V_0,B_0), (V_1,B_1) \in \mathscr{I}(\mathbb{F})$ are Witt-equivalent, then $(V_1 \oplus V_0,B_1 \oplus -B_0)$ is metabolic. The metabolic subspace $W \subset V_1 \oplus V_0$ is invariant under the isometry $I_1 \oplus I_0: V_1 \oplus V_0 \to V_1 \oplus V_0$, where $I_i:V_i \to V_i$ is the identity map for $i=0,1$. Then $(V_1 \oplus V_0,B_1 \oplus -B_0, I_1 \oplus I_0)$ is metabolic as an $\mathbb{F}$-directed isometric structure and we conclude that the function $\mathscr{I}(\mathbb{F}) \to \mathscr{VG}_{t-1}$ is well-defined. The map is also injective, since the definitions of metabolic coincide for $\mathscr{I}(\mathbb{F})$ and $\mathscr{VG}_{t-1}$. 
\newline
\newline
It therefore remains to show that $\mathscr{I}(\mathbb{F}) \to \mathscr{VG}_{t-1}$ is surjective. Let $(V,B,S)$ be an $\mathbb{F}$-directed isometric structure with $\Delta_S(t)=(t-1)^{2n}$. By Levine \cite{levine_2}, Lemma 12, either $(V,B,S)$ is metabolic, or it is concordant to an isometric structure $(V',B',S')$ such that the minimal polynomial of $S'$ is $t-1$. Then it follows that $S'=I$ and hence the map  $\mathscr{I}(\mathbb{F}) \to \mathscr{VG}_{t-1}$ is surjective. 
\end{proof}

Combining these two previous lemmas, we can now prove Theorem \ref{thm_classify_VG}.

\setcounter{thm}{0}

\begin{thm} For a field $\mathbb{F}$ of characteristic $\chi(\mathbb{F}) \ne 2$, $\mathscr{VG}^{\mathbb{F}} \cong \mathscr{I}(\mathbb{F}) \oplus \mathscr{G}^{\mathbb{F}}$.
\end{thm}

\begin{proof} By Theorem \ref{thm_eta_iso}, $\mathscr{VG}^{\mathbb{F}} \cong \mathscr{VG}_{\mathbb{F}}$. Combining Lemmas \ref{lemma_primary_decomposition} and \ref{lemma_t_minus_1_easy}, we have: 
\[
\mathscr{VG}_{\mathbb{F}} \cong \mathscr{VG}_{t-1} \oplus \bigoplus_{\lambda(t) \ne t \pm 1} \mathscr{VG}_{\lambda(t)} \cong \mathscr{I}(\mathbb{F}) \oplus \bigoplus_{\lambda(t) \ne t \pm 1} \mathscr{VG}_{\lambda(t)}.
\] 
A given $\mathbb{F}$-directed isometric structure $(V,B,S)$ can be decomposed as $(V_{t-1} \oplus V',B_{t-1} \oplus B', S_{t-1} \oplus S')$, where $(V_{t-1},B_{t-1},S_{t-1})$ is the restriction of $(V,B,S)$ to the $(t-1)$-primary component. Moreover, we have that $t-1$ does not divide $\Delta_{S'}(t)$. Since $t+1$ also does not divide $\Delta_{S'}(t)$, it follows that $\Delta_{S'}(-1)\Delta_{S'}(1) \ne 0$. Hence, $(V',B',S')$ is in $\mathscr{G}_{\mathbb{F}}$. Therefore, $\bigoplus_{\lambda(t) \ne t \pm 1} \mathscr{VG}_{\lambda(t)} \cong \bigoplus_{\lambda(t) \ne t \pm 1} \mathscr{G}_{\lambda(t)}$. Lastly, we have that $\mathscr{G}^{\mathbb{F}} \cong \mathscr{G}_{\mathbb{F}} \cong \bigoplus_{\lambda(t) \ne t \pm 1} \mathscr{G}_{\lambda(t)}$ and hence the proof is complete. 
\end{proof}

As an immediate corollary of Theorem \ref{thm_classify_VG}, we classify the possible orders of torsion in $\mathscr{VG}^{\mathbb{Q}}$.

\begin{corollary} \label{cor_orders} The only possible torsion in $\mathscr{VG}^{\mathbb{Q}}$ is of order $1,2,4$, or $\infty$. 
\end{corollary}
\begin{proof} By Theorem \ref{thm_classify_VG}, $\mathscr{VG}^{\mathbb{Q}}\cong \mathscr{I}(\mathbb{Q}) \oplus \mathscr{G}^{\mathbb{Q}}$. By Levine \cite{levine_2}, Proposition 22b, $\mathscr{G}^{\mathbb{Q}} \cong \mathscr{G}_{\mathbb{Q}}$ has torsion of orders only $1,2,4,\infty$. Recall that $\mathscr{W}(\mathbb{Q}) \cong \mathscr{W}(\mathbb{R}) \oplus \coprod_p \mathscr{W}(\mathbb{Z}/p\mathbb{Z})$, for $p \in \mathbb{N}$ prime (Scharlau \cite{scharlau}, Theorem 5.3.4). Elements of $\mathscr{W}(\mathbb{R})$ are classified by their signature. Hence, $\mathscr{W}(\mathbb{R}) \cong \mathbb{Z}$. For $p=2$, $\mathscr{W}(\mathbb{Z}/2\mathbb{Z})\cong \mathbb{Z}/2\mathbb{Z}$, the ring of two elements. For $p \equiv 1 \pmod 4$, $\mathscr{W}(\mathbb{Z}/p\mathbb{Z}) \cong \mathbb{Z}/2\mathbb{Z} \oplus \mathbb{Z}/2\mathbb{Z}$. For $p \equiv 3 \pmod{4}$, $\mathscr{W}(\mathbb{Z}/p\mathbb{Z})\cong \mathbb{Z}/4\mathbb{Z}$. Thus, every element of $\mathscr{VG}^{\mathbb{Q}}$ has order $1,2,4$, or $\infty$.
\end{proof}

\begin{remark} Corollary \ref{cor_orders} also holds over any field $\mathbb{F}$ which is a finite extension of the rational numbers. In particular, if $\mathbb{F}$ is a global field of characteristic $\chi(\mathbb{F})=0$, $\mathscr{VG}^{\mathbb{F}}$ has finite torsion only of orders $1$, $2$, and $4$. For a proof, see version 1 of the present paper on the arXiv.
\end{remark}

\subsection{The coupled algebraic concordance group} \label{sec_coupled} Next we define a group extension of $\mathscr{VG}^{\mathbb{F}}$ that obstructs concordance to pairs of classical $\mathbb{F}$-Seifert matrices. The definition is motivated by the properties enjoyed by pairs $A^{\pm}$ of Seifert matrices of knots in thickened surfaces (see Lemma \ref{lemma_seif_mat_facts}).

\begin{definition}[$\mathbb{F}$-Seifert couple] \label{defn_F_Seifert_couple} A pair $(A^+,A^-)$ of $2n \times 2n$ matrices with coefficients in $\mathbb{F}$ will be called an $\mathbb{F}$-Seifert couple if $A^--A^+$ is skew-symmetric and $\det(A^--A^+) \ne 0$. $\mathbb{F}$-Seifert couples will be denoted with boldface letters as follows: $\textbf{A}=(A^+,A^-)$.  An $\mathbb{F}$-Seifert couple is \emph{admissible} if $A^++(A^+)^{\intercal}=A^-+(A^-)^{\intercal}$ is non-singular (see Lemma \ref{lemma_seif_mat_facts}). An $\mathbb{F}$-Seifert couple is said to be \emph{metabolic} if $A^+,A^-$ are simultaneously metabolic. That is, there is a matrix $P$ with entries in $\mathbb{F}$ such that $\det(P) \ne 0$ and $PA^{\pm}P^{\intercal}$ has a half-dimensional block of zeros in the upper left corner. The operation $\oplus$ on $\mathbb{F}$-Seifert couples $\textbf{A}_0,\textbf{A}_1$ is defined by $\textbf{A}_0 \oplus \textbf{A}_1=(A_0^+\oplus A_1^+,A_0^-\oplus A_1^-)$. Two $\mathbb{F}$-Seifert couples $\textbf{A}_0,\textbf{A}_1$ are said to be \emph{concordant} if $\textbf{A}_1\oplus-\textbf{A}_0$ is a metabolic couple. We write $\textbf{A}_0 \sim \textbf{A}_1$ when $\textbf{A}_0,\textbf{A}_1$ are concordant.
\end{definition}

\begin{remark} By a $\mathbb{Z}$-Seifert couple, we mean a pair $(A^+,A^-)$ of $2n$-dimensional matrices over $\mathbb{Z}$ such that $A^--A^+$ is skew-symmetric and $\det(A^--A^+)=1$. In this case, all relations in the above definition are defined up to unimodular congruence. Theorem \ref{lemma_witt_cancel_0} also holds in this case and hence the concordance classes of $\mathbb{Z}$-Seifert couples form a group.
\end{remark}

To show that the concordance classes of $\mathbb{F}$-Seifert couples form an abelian group, we need the again need a corresponding Witt-cancellation property. This is accomplished by the lemma below.

\begin{lemma}\label{lemma_witt_cancel} Let $\textbf{N}$, $\textbf{A}$ be $\mathbb{F}$-Seifert couples of dimension $2k,2m$, respectively. Suppose that $\textbf{N}$ and $\textbf{A} \oplus \textbf{N}$ are metabolic. If $\textbf{N}$ is admissible, then $\textbf{A}$ is metabolic. 
\end{lemma}

\begin{proof} Here we will use the notation from the proof of Lemma \ref{lemma_witt_cancel_0}. Since $\textbf{N}$ is metabolic, there is a decomposition $\mathbb{F}^{2k} \cong \mathbb{F}^k \oplus \mathbb{F}^l$ such that $k=l$, $(v\oplus 0)^*N^+(\mathbb{F}^k\oplus 0)=(v\oplus 0)^*N^-(\mathbb{F}^k \oplus 0)=0$ for all $v \in \mathbb{F}^k$. Similarly, there is a subspace $V \subset \mathbb{F}^{2m} \oplus \mathbb{F}^k \oplus \mathbb{F}^l$ such that $v^*(A^+ \oplus N^+)(V)=v^*(A^- \oplus N^-)(V)=0$ for all $v \in V$. This uses the hypothesis that $N^{\pm}$ and $A^{\pm}\oplus N^{\pm}$ are simultaneously metabolic. As in the proof of Lemma \ref{lemma_witt_cancel_0}, we find a subspace $X \subset \mathbb{F}^{2m}$ such that $x^* A^+(X)=x^* A^-(X)=0$ for all $x \in X$. It needs only be shown that $\dim(X) \ge m$. Note that $N^++(N^+)^{\intercal}=N^-+(N^-)^{\intercal}$. This implies that the argument of Lemma \ref{lemma_witt_cancel_0} applies without alteration. Hence, $\dim(X)\ge m$. 
\end{proof}

As in the case of the virtual algebraic concordance group, Lemma \ref{lemma_witt_cancel} implies that $\sim$ is an equivalence relation on the set of admissible $\mathbb{F}$-Seifert couples and that $\oplus$ is a commutative and associative operation on $\sim$-equivalence classes (see again Levine \cite{levine}, Section 3). Denote by $\textbf{0}$ the concordance class of metabolic $\mathbb{F}$-Seifert couples.

\begin{definition}[The coupled algebraic concordance group] The group of $\sim$-equivalence classes of admissible $\mathbb{F}$-Seifert couples with operation $\oplus$ and additive identity $\textbf{0}$ is called the \emph{coupled algebraic concordance group} over $\mathbb{F}$. It is denoted $(\mathscr{VG},\mathscr{VG})^{\mathbb{F}}$. 
\end{definition}

The groups $\mathscr{G}^{\mathbb{F}}$, $\mathscr{VG}^{\mathbb{F}}$, and $(\mathscr{VG},\mathscr{VG})^{\mathbb{F}}$ are related as follows. First, there are the two projections: 
$$\pi^{\pm}: (\mathscr{VG},\mathscr{VG})^{\mathbb{F}} \to\mathscr{VG}^{\mathbb{F}}, \quad  \pi^{\pm}(A^+,A^-) = A^{\pm}.$$
Secondly, there is a map from the classical algebraic concordance group to the coupled one:
\[
\iota: \mathscr{G}^{\mathbb{F}} \to (\mathscr{VG},\mathscr{VG})^{\mathbb{F}}, \quad \iota(A)=(A,A^{\intercal})
\]
In the following lemmas, we show that $\mathscr{VG}^{\mathbb{F}}$ is a homomorphic image of $(\mathscr{VG},\mathscr{VG})^{\mathbb{F}}$ and $\mathscr{G}^{\mathbb{F}}$ embeds into $(\mathscr{VG},\mathscr{VG})^{\mathbb{F}}$. These results are then used to find an obstruction for a Seifert couple to be concordant to a Seifert couple of a classical knot.

\begin{lemma} \label{thm_GF_embeds} The map $\iota:\mathscr{G}^{\mathbb{F}} \to (\mathscr{VG},\mathscr{VG})^{\mathbb{F}}$ is an embedding.
\end{lemma}
\begin{proof} Since $A$ is an $\mathbb{F}$-Seifert matrix, $A+A^{\intercal}$ is non-singular. Since $A-A^{\intercal}$ is skew-symmetric, $(A,A^{\intercal})$ is an admissible $\mathbb{F}$-Seifert couple.  If $A \sim B$, then $(A \oplus -B,A^{\intercal} \oplus -B^{\intercal})$ is metabolic. In particular, if $A \sim 0$, then $(A,A^{\intercal}) \sim \textbf{0}$ and hence the map is an embedding. 
\end{proof}

\begin{lemma} The maps $\pi^{\pm}$ are surjective and $\mathscr{VG}^{\mathbb{F}}$ is a homomorphic image of $(\mathscr{VG},\mathscr{VG})^{\mathbb{F}}$.
\end{lemma}

\begin{proof} Every admissible $\mathbb{F}$-Seifert couple $(A^+,A^-)$ yields $\mathbb{F}$-directed matrices $A^+,A^-$ by definition. Definitions \ref{defn_F_Seifert_couple} and \ref{defn_uncoupled} imply that $\pi^{\pm}$ are well-defined. To show that $\pi^{\pm}$ are surjective, let $A$ be a $2k \times 2k$ dimensional $\mathbb{F}$-directed matrix and define:
\begin{equation}\label{defn_H}
H_{2k} = \begin{bmatrix}
0 & I_k \\
-I_k & 0
\end{bmatrix},
\end{equation}
where $I_k$ is the $k \times k$ identity matrix. Set $A^-=A+H_{2k}$ and $A^+=A$. Then $A^--A^+$ is skew-symmetric and $A^-+(A^-)^{\intercal}=A^++(A^+)^{\intercal}=A+A^{\intercal}$ is non-singular. Thus, $(A^+,A^-)$ is an admissible $\mathbb{F}$-Seifert couple and it follows that $\pi^+$ is surjective. Similarly, $\pi^-$ is surjective.
\end{proof}

\begin{proposition}\label{prop_g_in_eqlzr} The classical integral algebraic concordance group $\mathscr{G}^{\mathbb{\mathbb{Z}}}$ embeds into the equalizer $\text{Eq}(\pi^+,\pi^-) \subset (\mathscr{VG},\mathscr{VG})^{\mathbb{Q}}$ of $\pi^+$ and $\pi^-$. That is, $\pi^+(\iota(A)) \sim \pi^-(\iota(A))$ in $\mathscr{VG}^{\mathbb{Q}}$. In particular, if $\textbf{A}$ is an admissible $\mathbb{\mathbb{Q}}$-Seifert pair and $\pi^+(\textbf{A}) \not \sim \pi^-(\textbf{A})$ in $\mathscr{VG}^{\mathbb{\mathbb{Q}}}$, $(A^+,A^-)$ is not concordant to a Seifert couple of a classical knot in $S^3$.
\end{proposition}

\begin{proof} By Levine \cite{levine_2}, the inclusion map $\mathbb{Z} \to \mathbb{Q}$ gives an embedding $\mathscr{G}^{\mathbb{Z}} \to \mathscr{G}^{\mathbb{Q}}$. By Theorem \ref{thm_GF_embeds}, $\mathscr{G}^{\mathbb{Q}}$ embeds into $(\mathscr{VG},\mathscr{VG})^{\mathbb{Q}}$ via the map $A \to (A,A^{\intercal})$. If $A$ is a Seifert matrix of a knot $K$ in $S^3$, a Seifert matrix for the concordance inverse of $K$ is given by $-A^{\intercal}$. Thus, $A \oplus -A^{\intercal} \sim 0$ and we have that $A \sim A^{\intercal}$. This implies that $\pi^+(A,A^{\intercal}) \sim \pi^-(A,A^{\intercal})$ in $\mathscr{VG}^{\mathbb{Q}}$ and $(A,A^{\intercal})\in \text{Eq}(\pi^+,\pi^-)$.
\end{proof}

\subsection{$\mathbb{Z}/2\mathbb{Z}$-Quadratic forms $\&$ the Arf invariant} \label{sec_arf} Let $X$ be a vector space over $\mathbb{F}$. Recall that a quadratic form is a map $q:X \to \mathbb{F}$ such that $q(\alpha x)=\alpha^2q(x)$ for $\alpha \in \mathbb{F}$ and the map:
\[
B_q(x,y):=q(x+y)-q(x)-q(y)
\]
is a symmetric bilinear form. The form $B_q$ is called the bilinear form \emph{associated to $q$}. A quadratic form is said to be \emph{regular} if $B_q$ is regular, that is, if a matrix for $B_q$ is non-singular over $\mathbb{F}$. 
\newline
\newline
Given an $\mathbb{F}$-directed matrix $A$, there is a quadratic form $q_A:\mathbb{F}^{2n} \to \mathbb{F}$ defined by $q_A(x)=x^{\intercal}A x$. The following lemma shows that an $\mathbb{F}$-Seifert couple defines a unique quadratic form when $\mathbb{F}=\mathbb{Z}/2\mathbb{Z}$

\begin{lemma} Let $\mathbb{F}=\mathbb{Z}/2\mathbb{Z}$ and $(A^+,A^-)$ be an $\mathbb{F}$-Seifert couple. Then $q_{A^+}=q_{A^-}$
\end{lemma}

\begin{proof} Since $A^--A^+$ is skew-symmetric, we may assume after a change of basis that $A^--A^+=H$, where $H$ is defined as in Equation \ref{defn_H}. Note that for any $x \in \mathbb{F}^{2n}$, $x^{\intercal} H x \equiv 0 \pmod 2$. Hence, $x^{\intercal} A^- x \equiv x^{\intercal}A^+x \pmod{2}$ and the lemma is proved.
\end{proof}

\begin{definition}[Quadratic form of a Seifert couple] If $\textbf{A}=(A^+,A^-)$ is a $\mathbb{Z}/2\mathbb{Z}$-Seifert couple of dimension $2n$ its quadratic form is defined to be $q_{\textbf{A}}:=q_{A^+}=q_{A^-}$. 
\end{definition}

Next we recall some properties of the Arf invariant of a regular quadratic form $q:X \to \mathbb{Z}/2\mathbb{Z}$. The reader is referred to Scharlau \cite{scharlau}, Section 9.4, for more details. It is well known that the Arf invariant can be defined by the rule $\text{Arf}(q)=0$ if $q$ maps the majority of vectors in $X$ to $0$ and $\text{Arf}(q)=1$ if $q$ maps a majority of the vectors in $X$ to $1$. Denote by $\perp$ the orthogonal sum of quadratic spaces. Then the Arf invariant satisfies $\text{Arf}(q_1 \perp q_2)=\text{Arf}(q_1)+\text{Arf}(q_2)$ whenever $q_1:X_1 \to \mathbb{Z}/2\mathbb{Z}$ and $q_2:X_2 \to \mathbb{Z}/2\mathbb{Z}$ are regular. A quadratic form $q:X \to \mathbb{Z}/2\mathbb{Z}$ is said to be \emph{metabolic} if there is a half-dimensional subspace $W \subset X$ such that $q(W)=0$. If a regular quadratic space is metabolic, $\text{Arf}(q)=0$. For the Arf invariant of classical knots, see Kauffman \cite{on_knots}.

\begin{theorem} \label{thm_arf_obstructs} If $\textbf{A}_0,\textbf{A}_1$ are concordant $\mathbb{Z}/2\mathbb{Z}$-Seifert couples and $q_{\textbf{A}_0},q_{\textbf{A}_1}$ are regular, then:
$$\text{Arf}(q_{\textbf{A}_0})=\text{Arf}(q_{\textbf{A}_1}).$$
\end{theorem}

\begin{proof} By hypothesis, $A_1^+ \oplus -A_0^+$ is metabolic. Hence, $q_{A_1^+} \perp q_{A_0^+}$ is metabolic and $\text{Arf}(q_{A_1^+} \perp q_{A_0^+})=0$. The result now follows from the additivity of the Arf invariant.
\end{proof}

If $\textbf{A}$ is a $\mathbb{Z}$-Seifert couple of a Seifert surface $F \subset \Sigma \times I$, then the bilinear form associated to $q_{\textbf{A}}$ has a simple geometric interpretation. First, define a quadratic form $q_F$ on $H_1(F;\mathbb{Z}/2\mathbb{Z})$ by:
\[
q^{\pm}_{F}(x)=\theta^{\pm}(x,x)=\lk_{\Sigma}(x^{\pm},x) \pmod{2}.
\]
Then $q_F=q_{\textbf{A}}$ for some appropriately chosen basis. Denote the intersection forms on $F$, $\Sigma$ with the binary operations $\bullet_{F}$, $\bullet_{\Sigma}$, respectively. Define $\star_F:H_1(F;\mathbb{Z}/2\mathbb{Z}) \times H_1(F;\mathbb{Z}/2\mathbb{Z}) \to \mathbb{F}_2$ by:
\[
x\star_F y= x\bullet_F y+p_*(x)\bullet_{\Sigma} p_*(y) \pmod{2},
\]
where $p:\Sigma \times I \to \Sigma$ is projection onto the first factor. Then $\star_F$ is a symmetric bilinear form over $\mathbb{Z}/2\mathbb{Z}$. It is important to note that $\star_F$ is not necessarily regular. 

\begin{proposition} The form $\star_F$ is the bilinear form associated to the quadratic form $q_F=q_{\textbf{A}}$. 
\end{proposition} 
\begin{proof} The result follows from the following calculation.
\begin{align*}
q_{F}(x+y)&=\theta^{-}(x+y,x+y) \\
                  &=\theta^{-}(x,x)+\theta^-(x,y)+\theta^-(y,x)+\theta^-(y,y) \\
                  &=q_{F}(x)+q_{F}(y)+\lk_{\Sigma}(x^-,y)+\lk_{\Sigma}(y^-,x) \\
                  &=q_{F}(x)+q_{F}(y)+\lk_{\Sigma}(x,y^+)+\lk_{\Sigma}(y^-,x) \\
                  &=q_{F}(x)+q_{F}(y)+\lk_{\Sigma}(x,y^+)+p_*(x)\bullet_{\Sigma} p_*(y)+\lk_{\Sigma}(x,y^-)\\
                  &\equiv q_{F}(x)+q_{F}(y)+\lk_{\Sigma}(x,y^+)-\lk_{\Sigma}(x,y^-)+p_*(x)\bullet_{\Sigma} p_*(y) & \pmod{2} \quad\quad\quad \\
                  &\equiv q_{F}(x)+q_{F}(y)+ x \bullet_{F} y+p_*(x)\bullet_{\Sigma} p_*(y) & \pmod{2} \quad\quad\quad \\
                  & \equiv q_{F}(x)+q_{F}(y)+x\star_F y & \pmod{2}  \quad\quad\quad
\end{align*}
The fifth step above uses Equation (\ref{eqn_ct}). The next-to-last step uses the fact that $y^+-y^-=\partial z$ for some $2$-chain $z$ and hence $\lk_{\Sigma}(x,y^+)-\lk_{\Sigma}(x,y^-) \equiv x \bullet_{\Sigma \times I} z \equiv x \bullet_F y \pmod{2}$.
\end{proof}

\section{Geometric concordance} \label{sec_gc}

\subsection{Virtual Concordance of Seifert surfaces} \label{sec_virt_conc_defn} In \cite{bcg2}, it was proved that every $\mathbb{Z}$-Seifert couple is realized by a Seifert surface of some almost classical knot. The goal of this section is to define a geometric relation on Seifert surfaces, called \emph{virtual concordance}, that induces algebraically concordant $\mathbb{Z}$-Seifert couples. Following the definition, we give two concrete examples of virtually concordant Seifert surfaces. Theorem \ref{thm_B}, which relate algebraic and geometric concordance, is proved in Section \ref{sec_proof_thm_B}. A given knot $K \subset \Sigma \times I$ bounds many Seifert surfaces that are not virtually concordant. In Section \ref{sec_ambient_Z}, we show how to obtain all the virtual concordance classes of Seifert surfaces of $K$ from any Seifert surface of $K$. 

\begin{definition}[Virtual concordance of Seifert surfaces] \label{defn_virt_conc_Seif} Let $\Sigma_0,\Sigma_1$ be closed oriented surfaces. Two Seifert surfaces $F_0 \subset \Sigma_0 \times I, F_1 \subset \Sigma_1\times I$ of knots $\partial F_0=K_0,\partial F_1=K_1$ will be called \emph{virtually concordant} if there is a compact oriented $3$-manifold $W$, a properly embedded annulus $A$ in $W \times I$, and a compact, oriented, and two-sided $3$-manifold $M \subset W \times I$ such that $\partial W=\Sigma_1 \sqcup -\Sigma_0$, $\partial A=K_1 \sqcup -K_0$, $\partial M=F_1 \cup A \cup -F_0$ and $M \cap \partial (W \times I)= F_1 \sqcup -F_0$. A Seifert surface $F \subset \Sigma \times I$ that is concordant to a disc $D^2 \subset S^2 \times I$, will be called \emph{slice}.
\end{definition}

Every Seifert surface $F \subset \Sigma \times I$ is virtually concordant to itself, as is seen by setting $W=\Sigma \times I$, $M=F \times I$, and $A=\partial F \times I$. The relation can be easily seen to be symmetric and transitive, so that virtual concordance is an equivalence relation on Seifert surfaces.
\newline
\newline
In \cite{myers_new}, Myers defines two Seifert surfaces $F,F' \subset S^3$ to be concordant if there is an embedding $h:F \times I \to S^3 \times I$ such that $h(F \times 0)=F=h(F \times I) \cap (S^3 \times 0)$, $h(F \times 1)=F'=h(F \times I) \cap (S^3 \times 1)$, and $\partial (S^3 \times I) \cap \partial h(F \times I)=F \cup -F'$. This differs from virtual concordance in two ways. First, the cobordism $M$ in Definition \ref{defn_virt_conc_Seif} need not be a cylinder $F \times I$. Secondly, a cobordism movie of a concordance between two classical Seifert surfaces in $S^3$ has the same ambient space at each time $t$. Every frame of the movie is drawn in $S^3 \times t$ for some $t \in [0,1]$. On the other hand, the ambient space of a time slice in a virtual concordance need not be constant. If $W$ is as in Definition \ref{defn_virt_conc_Seif} and $f:W \to I$ is a Morse function with $f^{-1}(0)=\Sigma_0$ and $f^{-1}(1)=\Sigma_1$, a time slice is of the form $f^{-1}(t) \times I$ for some regular value $t \in I$. The difference between a constant and non-constant ambient space is illustrated in the two examples below. A third example appears in Section \ref{sec_vss}.

\begin{example} \label{example_first_virt_conc} A movie of a virtual concordance between two discs in $S^2 \times I$ is shown in Figure \ref{fig_v_conc_discs}. The gray background in each frame represents a neighborhood of $S^2$. Set $W=S^2 \times I$. The frames of the movie depict cross-sections $S^2 \times t$ of $W$ where $t=0,\tfrac{1}{2},1$. The discs are colored cyan, where $D_0 \subset (S^2 \times 0) \times I$ appears in the $t=0$ frame, and $D_1 \subset (S^2 \times 1) \times I$ appears in the $t=1$ frame. Note that the coordinate ``$\times I$'' of $W \times I$ will be visible only through the projection $(S^2 \times t) \times I \to S^2 \times t$ onto the cross-section at $t$. The cobordism $M \subset W \times I$ passes through a saddle at $t=\frac{1}{4}$ and a death at $t=\tfrac{3}{4}$. The death can be visualized as the small gray circle shrinking to a point and then disappearing. The $3$-manifold $M$ has boundary $D_1 \cup -D_0 \cup A$ where $A$ is an annulus. The cobordism $M$ between $D_0$ and $D_1$ is a $3$-ball $D^2 \times I \subset W \times I$. \hfill $\square$
\end{example}

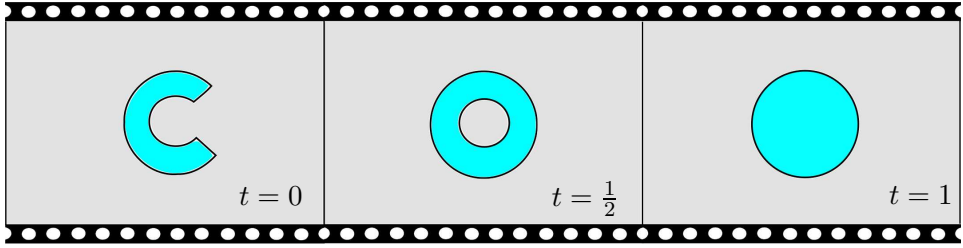
\begin{figure}[htb]
    \centering
    \def\svgwidth{5in}
\begingroup%
  \makeatletter%
  \providecommand\color[2][]{%
    \errmessage{(Inkscape) Color is used for the text in Inkscape, but the package 'color.sty' is not loaded}%
    \renewcommand\color[2][]{}%
  }%
  \providecommand\transparent[1]{%
    \errmessage{(Inkscape) Transparency is used (non-zero) for the text in Inkscape, but the package 'transparent.sty' is not loaded}%
    \renewcommand\transparent[1]{}%
  }%
  \providecommand\rotatebox[2]{#2}%
  \newcommand*\fsize{\dimexpr\f@size pt\relax}%
  \newcommand*\lineheight[1]{\fontsize{\fsize}{#1\fsize}\selectfont}%
  \ifx\svgwidth\undefined%
    \setlength{\unitlength}{516.68006181bp}%
    \ifx\svgscale\undefined%
      \relax%
    \else%
      \setlength{\unitlength}{\unitlength * \real{\svgscale}}%
    \fi%
  \else%
    \setlength{\unitlength}{\svgwidth}%
  \fi%
  \global\let\svgwidth\undefined%
  \global\let\svgscale\undefined%
  \makeatother%
  \begin{picture}(1,0.25188995)%
    \lineheight{1}%
    \setlength\tabcolsep{0pt}%
    \put(0,0){\includegraphics[width=\unitlength]{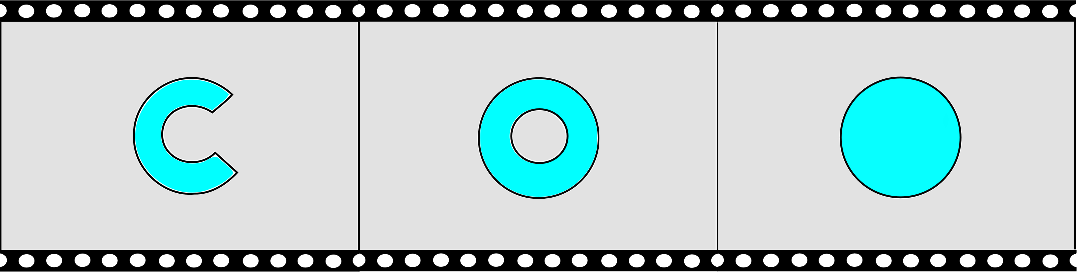}}%
    \put(0.2445374,0.04142458){\color[rgb]{0,0,0}\makebox(0,0)[lt]{\lineheight{40.54999924}\smash{\begin{tabular}[t]{l}$t=0$\end{tabular}}}}%
    \put(0.57213242,0.04015977){\color[rgb]{0,0,0}\makebox(0,0)[lt]{\lineheight{40.54999924}\smash{\begin{tabular}[t]{l}$t=\tfrac{1}{2}$\end{tabular}}}}%
    \put(0.92502421,0.04268944){\color[rgb]{0,0,0}\makebox(0,0)[lt]{\lineheight{40.54999924}\smash{\begin{tabular}[t]{l}$t=1$\end{tabular}}}}%
  \end{picture}%
\endgroup%

    \caption{A virtual concordance between two discs $D_0 \subset S^2 \times I$ ($t=0$) and $D_1 \subset S^2 \times I$ ($t=1$). See Example \ref{example_first_virt_conc}.}
    \label{fig_v_conc_discs}
\end{figure}

\begin{example}  \label{example_second_virt_conc} For any Seifert surface $F_1 \subset \Sigma_1 \times I$, a non-trivial virtual concordance can be constructed as follows. See Figure \ref{fig_conc_movie}, $t=1$. Here we have drawn a small band $\approx I \times I$ (cyan) in $F_1$ lying above a small ball (gray) contained in $\Sigma_1$. At $t=0$, we see a (gray) surface $\Sigma_0$ obtained from $\Sigma_1$ by attaching a $1$-handle. The $1$-handle can be seen by identifying the two circular ``cigarette burns''\footnote{Films shown in movie theaters have ``cigarette burns'' in the upper right hand corner that indicate the time at which the projectionist should change reels.} in the filmstrip. Note that this is essentially the standard Heegard diagram notation (see Gompf-Stipsicz \cite{gs}). In the $t=0$ frame, the band (cyan) is knotted into a tangle that passes over the added $1$-handle. Outside the depicted region, the surface $F_0$ is identical to $F_1$. To show that $F_0$ and $F_1$ are virtually concordant, we draw a movie which shows both the $3$-manifolds $M$ and $W$ that are required by Definition \ref{defn_virt_conc_Seif}. In Figure \ref{fig_conc_movie}, the cobordism $M$ between $F_0$ and $F_1$ is colored in cyan while the ambient space $W$ is drawn in gray. As in Example \ref{example_first_virt_conc}, the ``thickened'' factor of $W \times I$ is not drawn and instead the cross-sections of $M$ are projected onto the cross-sections of $W$. 
\newline
\newline
First we describe how Figure \ref{fig_conc_movie} depicts the cobordism $W$. The space $W$ is given by $W=(\Sigma_0 \times [0,\tfrac{1}{2}]) \cup (2\text{-handle}) \cup (\Sigma_1 \times [\tfrac{1}{2},1])$. In Figure \ref{fig_conc_movie}, the frames with $0 \le t < \tfrac{1}{2}$ show the surface $\Sigma_0$ and those with $\tfrac{1}{2} <t \le 1$ show the surface $\Sigma_1$. At $t=\tfrac{1}{2}$, the magenta curve represents the attaching sphere of a $3$-dimensional $2$-handle. Note we are again using the standard Heegard diagram notation to depict the attaching sphere as a curve on $\Sigma_0$. This added $2$-handle cancels the $1$-handle. Hence, $\partial W=\Sigma_1 \sqcup -\Sigma_0$.
\newline
\newline
Next we describe the cobordism $M$ between $F_0$ and $F_1$. At $t=0$, $M$ is $F_0$. On the interval $\left[0,\frac{1}{4}\right]$, the movie shows the addition of a $3$-dimensional $1$-handle to $F_0 \times [0,\varepsilon]$ for some $0<\varepsilon <\frac{1}{4}$. This can be most easily visualized by thickening the graph of standard saddle $z=x^2-y^2$ in $\mathbb{R}^3$. After an isotopy, we obtain the frame at $t=\frac{3}{8}$. The cross-section of $M$ then remains  unchanged up to $t=\frac{5}{8}$. At $t=\frac{5}{8}$, a $3$-dimensional $2$-handle is added. This is visualized as thickening the graph of a standard $2$-handle $z=-x^2-y^2$ in $\mathbb{R}^3$. Over time, we see the interior circle shrink to a point. Then the exterior circle shrinks to a point, leaving just $F_1$ at $t=1$. Note that the movie shows that the knots $K_0=\partial F_0$ and $K_1=\partial F_1$ are connected by an annulus in $W \times I$. Indeed, the boundaries of the cross-sections of $M$ undergo two saddles and two deaths between $t=0$ and $t=1$. Thus, $F_0$ and $F_1$ are virtually concordant. \hfill $\square$
\end{example}

\begin{figure}[htb]
\def\svgwidth{5in}\tiny
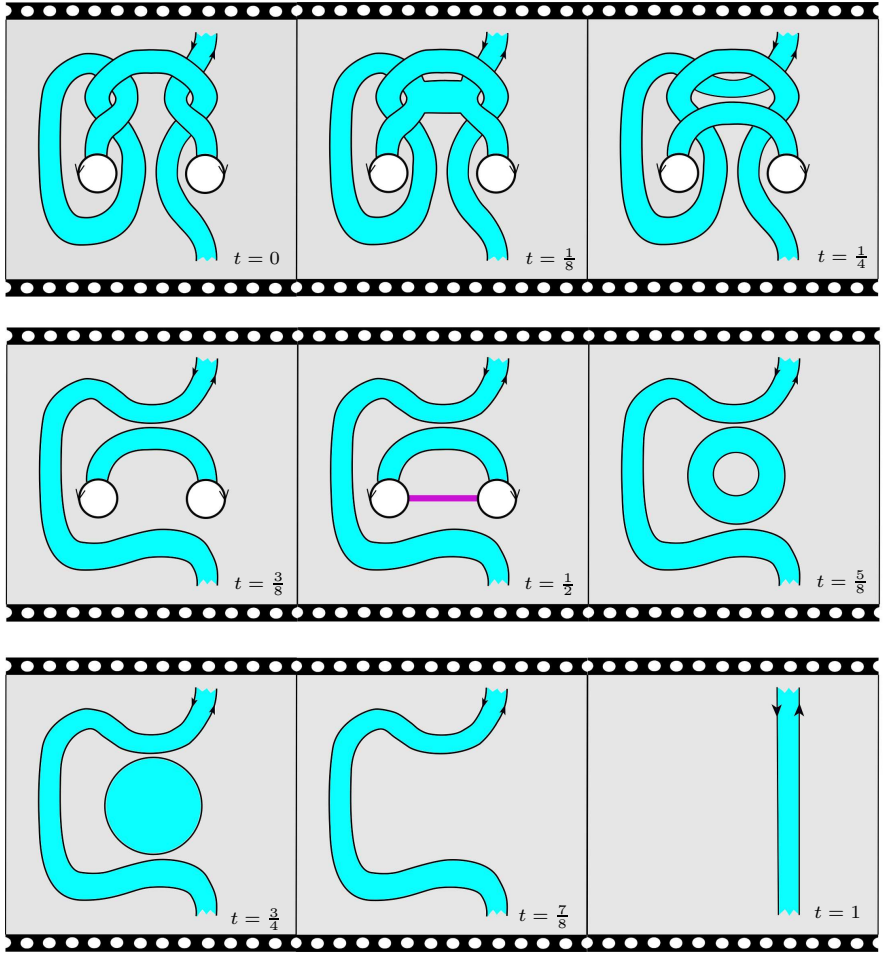
\normalsize
\caption{A virtual concordance between Seifert surfaces $F_0 \subset \Sigma_0 \times I$ ($t=0$) and $F_1 \subset \Sigma_1 \times I$ ($t=1$). See Example \ref{example_second_virt_conc}.} \label{fig_conc_movie}
\end{figure}

\subsection{Proof of Theorem \ref{thm_B}} \label{sec_proof_thm_B} The goal of this subsection is to prove that virtually concordant Seifert surfaces have concordant $\mathbb{Z}$-Seifert couples. An immediate consequence of Theorem \ref{thm_B} is that our Arf invariant and the directed signature functions of \cite{bcg2} are virtual concordance invariants of Seifert surfaces.  Our proof is based on Cimasoni-Turaev \cite{ct}, Lemma 5.1.

\begin{thm}\label{thm_conc_imp_meta} If $F_0 \subset \Sigma_0 \times I, F_1 \subset \Sigma_1 \times I$ are virtually concordant Seifert surfaces with $\mathbb{Z}$-Seifert couples $\textbf{A}_0, \textbf{A}_1$, respectively, then $\textbf{A}_0,\textbf{A}_1$ are concordant in $(\mathscr{VG},\mathscr{VG})^{\mathbb{Z}}$.
\end{thm}

\begin{proof} By hypothesis, there is a compact oriented $3$-manifold $W$, an annulus $A$ properly embedded in $W \times I$, and a compact, oriented, 2-sided  $3$-manifold $M$ in $W \times I$ such that $\partial W=\Sigma_1 \sqcup -\Sigma_0$, $\partial A=\partial F_1 \sqcup -\partial F_0$, $\partial M=F_1 \cup A \cup -F_0$, and $M \cap \partial (W \times I)=F_1 \sqcup -F_0$. By Lickorish \cite{lick_book}, Corollary 8.16, there is a basis $[\alpha_1],\ldots,[\alpha_{2g}]$ over $\mathbb{Z}$ for $H_1(\partial M;\mathbb{Z})$ such that $[\alpha_1],\ldots,[\alpha_g]$ map to zero in $H_1(M;\mathbb{Q})$. Here $g=g_0+g_1$ where $g_i=\text{genus}(F_i)$. Since $\mathbb{Z}$ is an integral domain and the directed Seifert forms for $F_0,F_1$ are bilinear, it is sufficient to prove the theorem in the case that $[\alpha_1],\ldots,[\alpha_g]$ themselves map to zero in $H_1(M;\mathbb{Z})$. Integer coefficients will now be assumed for the remainder of the proof. Since $H_1(\partial M)\cong H_1(F_0)\oplus H_1(F_1)$, we may write $[\alpha_i]=[\alpha_i']\oplus[\alpha_i'']$ for some simple closed curves $\alpha_i'$ on $F_0$ and $\alpha_i''$ on $F_1$. Let $\theta_i^{\pm}$ denote the directed Seifert forms on $F_i$ and define $\theta^{\pm}:(H_1(F_0)\oplus H_1(F_1))\times (H_1(F_0)\oplus H_1(F_1)) \to \mathbb{Z}$ by:
\begin{align*}
\theta^{\pm}(x_0 \oplus x_1,y_0 \oplus y_1) &=\theta_1^{\pm}(x_1,y_1)-\theta_0^{\pm}(x_0,y_0)=\lk_{\Sigma_1}(x_1^{\pm},y_1)-\lk_{\Sigma_0}(x_0^{\pm},y_0).
\end{align*}
Let $x_0 \oplus x_1,y_0 \oplus y_1 \in H_1(F_0)\oplus H_1(F_1)$. Let $B_i$ be a 2-chain in $\Sigma_i \times I$ cobounded by $y_i$ and a collection $z_i$ of oriented simple closed curves in $\Sigma_i \times 1$. It follows from the definition of $\lk_{\Sigma}$ that:
\[
\lk_{\Sigma_i}(x_i^{\pm},y_i)=x_i^{\pm} \bullet_{\Sigma_i \times I} B_i,
\]
where $-\bullet_{\Sigma_i \times I}-$ denotes the algebraic intersection number in $\Sigma_i \times I$. Then we have that:
\[
\theta^{\pm}(x_0 \oplus x_1,y_0\oplus y_1) =x_1^{\pm}\bullet_{\Sigma_1 \times I} B_1-x_0^{\pm}\bullet_{\Sigma_0 \times I} B_0
=(x_1^{\pm}+x_0^{\pm})\bullet_{\partial (W \times I)} (B_1+B_0).
\]
The above follows from the fact that $\Sigma_1 \times I \sqcup -\Sigma_0 \times I$ is in $\partial (W \times I)$ and $x_i^{\pm}$ has no geometric intersections with $B_j$ for $i \ne j$. Also note that since $M$ is 2-sided in $W \times I$, the maps $\pm:F_i \to\Sigma_i \times I\smallsetminus F_i$, which push in the $\pm$-normal direction, extend to maps $\pm:M \to W \times I \smallsetminus M$.
\newline
\newline
Next we apply the above calculation to the bases for $H_1(F_0),H_1(F_1)$. As $[\alpha_i]=0$ in $H_1(M)$, there is a a locally flat surface $\varphi_i$ in $M$ such that $[\partial \varphi_i]=[\alpha_i]=[\alpha_i']\oplus[\alpha_i'']$. Then we have:
\[
\theta^{\pm}(\alpha_i,\alpha_j)=((\alpha_i'')^{\pm}+(\alpha_i')^{\pm})\bullet_{\partial (W \times I)} (B_j''+B_j'),
\]
where $B_j' \subset \Sigma_0 \times I$ is a 2-chain cobounded by $\alpha_j'$ and some curves in $\Sigma_0 \times 1$ and $B_j''\subset\Sigma_1 \times I$ is a 2-chain cobounded by $\alpha_j''$ and some curves in $\Sigma_1 \times 1$. Consider $B_j''+B_j'-\varphi_j$. This is a relative $2$-cycle in $H_2(W \times I,W \times 1) \cong 0$. Thus, there is a 3-chain $\tau_j$ such that $\partial \tau_j-(B_j''+B_j'-\varphi_j)=\omega_j$ for some 2-chain $\omega_j$ in $W \times 1$. Hence, $[B_j''+B_j'-\varphi_j+\omega_j]=0$ in $H_2(W \times I)$.
\newline
\newline
Before proceeding with the proof, we make a general remark about algebraic intersection numbers. Recall that for any orientable surface $S$ in $W\times I$ with boundary in $\partial (W \times I)$ and an orientable surface $T \subset \partial W \times I$, we have that $\partial S \bullet_{\partial (W \times I)} T=S\bullet_{W \times I} T$. This can be seen by pushing a neighborhood in $T$ of each transverse intersection of $\partial S$ with $T$ into $W \times I$ so that $S$ and $T$ intersect in a point (see e.g. Freedman-Quinn \cite{freedman_quinn}, page 12).
\newline
\newline
Continuing with the proof, recall that $M$ intersects $\partial (W \times I)$ only in $F_1 \sqcup -F_0$. It follows that the push-offs $\varphi^{\pm}_i$ have no geometric intersections with $\omega_j$ for all $i,j$. Combining this fact with the observations from the previous paragraphs, it follows that:
\[
\theta^{\pm}(\alpha_i,\alpha_j) =\varphi_i^{\pm} \bullet_{W \times I} (B_j''+B_j')=\varphi_i^{\pm} \bullet_{W \times I} (\varphi_j-\omega_j)=\varphi_i^{\pm} \bullet_{W \times I} \varphi_j.
\]
Finally, since $\varphi_j \subset M$ and $\varphi_i^{\pm} \subset W \times I \smallsetminus M$, there are no geometric intersections of $\varphi_i^{\pm}$ and $\varphi_j$ for $1 \le i,j \le g$. This implies that $\theta^{\pm}(\alpha_i,\alpha_j)=0$ for  $1 \le i,j \le g$. It follows that the pair $(A^+_1\oplus-A_0^+,A_1^-\oplus -A_0^-)$ is a metabolic $\mathbb{Z}$-Seifert couple.
\end{proof}

\begin{corollary} Suppose $F_0 \subset \Sigma_0 \times I, F_1 \subset \Sigma_1 \times I$ are Seifert surfaces with corresponding $\mathbb{Z}/2\mathbb{Z}$-Seifert couples $\textbf{A}_0,\textbf{A}_1$. If $q_{\textbf{A}_0},q_{\textbf{A}_1}$ are regular and $F_0,F_1$ are virtually concordant, then: 
\[
\text{Arf}(q_{\textbf{A}_0})=\text{Arf}(q_{\textbf{A}_0}).
\]
\end{corollary}

\begin{proof} This follows from Theorem \ref{thm_B} and Theorem \ref{thm_arf_obstructs}.
\end{proof}

Similarly, the directed signature functions of \cite{bcg2} are virtual concordance invariants of Seifert surfaces. Let $\omega \in S^1\smallsetminus 1$ be a complex number and suppose that $\Delta_{K,F}^{\pm}(\omega)\ne 0$. Then $(1-\omega) A^{\pm}+(1-\bar{\omega})(A^{\pm})^{\intercal}$ is a non-singular Hermitian matrix. The \emph{directed signature function} defined by:
\[
\widehat{\sigma}_{\omega}^{\pm}(K,F)=\text{sign}((1-\omega) A^{\pm}+(1-\bar{\omega})(A^{\pm})^{\intercal}).
\]
\begin{corollary} If $F_0\subset \Sigma_0 \times I,F_1 \subset \Sigma_1 \times I$ are virtually concordant Seifert surfaces of knots $K_0=\partial F_0$, $K_1=\partial F_1$, and $\omega \in S^1 \smallsetminus 1$ satisfies $\Delta_{K_0,F_0}^{\pm} (\omega) \Delta^{\pm}_{K_1,F_1}(\omega) \ne 0$, then:
$$\widehat{\sigma}_{\omega}^{\pm}(K_0,F_0)=\widehat{\sigma}_{\omega}^{\pm}(K_1,F_1).$$
\end{corollary}
\begin{proof} Let $A_0^{\pm}$, $A_1^{\pm}$ be directed Seifert matrices for $F_0,F_1$, respectively. Since $\Delta_{K_0,F_0}^{\pm} (\omega)\ne 0$ and  $\Delta^{\pm}_{K_1,F_1}(\omega) \ne 0$, the directed signature function is defined at $\omega$. By Theorem \ref{thm_conc_imp_meta}, $(A_1^{+} \oplus -A_0^{+},A_1^{-} \oplus -A_0^{-})$ is metabolic. Since $\Delta_{K_0,F_0}^{\pm} (\omega) \Delta^{\pm}_{K_1,F_1}(\omega) \ne 0$, 
$(1-\omega)(A_1^{\pm}\oplus -A_0^{\pm})+(1-\bar{\omega})(A_1^{\pm}\oplus -A_0^{\pm})^{\intercal}$
is non-singular. These Hermitian matrices thus have signature $0$. The additivity of the signature function implies that $0=\widehat{\sigma}_{\omega}^{\pm}(K_1,F_1)-\widehat{\sigma}_{\omega}^{\pm}(K_0,F_0)$ and the result follows.
\end{proof}

\subsection{Changing the Seifert surface of a knot} \label{sec_ambient_Z} Two Seifert surfaces of a fixed homologically trivial knot in a $3$-manifold $M$ are said to be \emph{$S$-equivalent} if they may be obtained from one another by a finite sequence consisting of three operations: (1) ambient isotopy rel $K$, (2) addition/removal of a compressible $1$-handle and (3) addition/removal of a $2$-sphere bounding $3$-ball in $M$. Any two Seifert surfaces of a knot in $S^3$ are $S$-equivalent, but this is not true in general. For the case of $M=\Sigma \times I$, suppose $K \subset \Sigma \times I$ is homologically trivial. The long exact sequence of the pair $(\Sigma \times I,K)$ gives $H_2(\Sigma \times I,K) \cong H_1(K) \oplus H_2(\Sigma)$. Clearly, $S$-equivalent Seifert surfaces of $K$ must represent the same element of $H_2(\Sigma \times I,K)$. Since $H_2(\Sigma) \cong \mathbb{Z}$, there are infinitely many Seifert surfaces of $K$ that are not $S$-equivalent. However, $S$-equivalence classes can be classified as follows.

\begin{lemma} Seifert surfaces $F_0,F_1 \subset \Sigma \times I$ of $K$ are $S$-equivalent if and only if $[F_0]=[F_1] \in H_2(\Sigma \times I,K)$. In other words, $S$-equivalence is classified by $H_2(\Sigma \times I, K) \cong H_1(K) \oplus H_2(\Sigma)$.
\end{lemma}

\begin{proof} This follows from more general results already in the literature. See Moishezon-Mandelbaum \cite{MM}, Lemma 4 and Proposition 6, or Kaiser \cite{UK}, Theorems 4.1.5 and 5.1.11.
\end{proof}

Let $[\Sigma]$ denote the class of $\Sigma \times 1$ in $H_2(\Sigma \times I,K)$. If $F,F'$ are two Seifert surfaces of $K$, then $[F']-[F]=n \cdot [\Sigma] \in H_2(\Sigma \times I,K)$ for some $n \in \mathbb{Z}$. Since two Seifert surfaces of $K$ representing the same element of $H_2(\Sigma \times I,K)$ are $S$-equivalent, it follows that $F'$ is $S$-equivalent to a surface obtained from $F$ by taking the disjoint union with $|n|$ parallel copies of $ \pm\Sigma$, where the sign is determined by the sign of $n$. A surface constructed in this way from a Seifert surface $F \subset \Sigma \times I$ will be denoted by $F+n \cdot \Sigma$. One can ensure that $F+n \cdot \Sigma$ is a connected surface by joining $F$ and each parallel copy $\pm\Sigma \times \varepsilon \subset \Sigma \times I$ of $\Sigma$ by some compressible $1$-handle. The next result describes the effect of changing the Seifert surface $F$ of a knot on the virtual concordance class of $F$.

\begin{theorem} \label{thm_ambient_Z} Let $F_0,F_0' \subset \Sigma_0 \times I$ and $F_1 \subset \Sigma_1 \times I$ be Seifert surfaces and $n \in \mathbb{Z}$.
\begin{enumerate}
    \item If $F_0,F_0'$ are $S$-equivalent, then they are virtually concordant.
    \item If $F_0,F_1$ are virtually concordant, then $F_0+n\cdot\Sigma_0$ and $F_1+n\cdot \Sigma_1$ are virtually concordant.
\end{enumerate}
\end{theorem}

\begin{proof} The first claim follows immediately from the definitions. For the second claim, it is sufficient to show that $F_0+\Sigma_0$ and $F_1+\Sigma_1$ are virtually concordant. By the preceding remarks, it may be assumed that $F_i+\Sigma_i$ is obtained from $F_i$ by connecting $F_i$ to a parallel disjoint copy of $\Sigma_i \times 1$ by a single compressible $1$-handle $\tau_i$. The first claim implies that the choice of $\tau_i$ is immaterial, since any two choices are obtained by an $S$-equivalence.   
\newline
\newline
By definition, there exists a $3$-manifold $W$, an annulus $A$ in $W \times I$, and a $3$-manifold $M$ in $W \times I$ such that $\partial W=\Sigma_1 \sqcup -\Sigma_0$, $\partial M=F_1 \cup A \cup -F_0$, and $\partial A=\partial F_1 \sqcup -\partial F_0$. We may also assume that the parallel copies of $\Sigma_0 \times 1,\Sigma_1 \times 1$ in $F_0+\Sigma_0,F_1+\Sigma_1$ are at the same height $\varepsilon \in (0,1) \subset I$ in $\Sigma_0 \times I$, $\Sigma_1 \times I$, respectively. Since $M$ is compact, there is an $\varepsilon'>0$ such that $\varepsilon\le \varepsilon'<1$ and $M \cap (W \times \varepsilon')=\varnothing$. Applying an isotopy to $F_i+\Sigma_i$ that fixes $F_i$, it may be assumed that $\varepsilon=\varepsilon'$.
\newline
\newline
A cobordism between $F_0+\Sigma_0$ and $F_1+\Sigma_1$ may now be constructed as follows. First attach a $3$-dimensional $2$-handle to the belt sphere of $\tau_i$. More precisely, we attach to $W$ collars  $C_0=\Sigma_0 \times [0,1]$ and $C_1=\Sigma_1 \times [0,1]$. In $C_i \times I$, there is a cobordism $M_i$ between $F_i+\Sigma_i \subset \Sigma_i \times 0 \times I$ and $F_i \sqcup \Sigma_i \subset \Sigma_i \times 1 \times I$ that corresponds to attaching a $2$-handle $h_i^2$ along the belt sphere of $\tau_i$ in $C_i\times \tfrac{1}{2} \times I$. The cobordism between $F_0+\Sigma_0$ and $F_1+\Sigma_1$ is then $M_0 \cup (M \sqcup W \times \varepsilon) \cup M_1$. This cobordism lies in $(C_0 \cup W \cup C_1) \times I$ and satisfies the requirements of Definition \ref{defn_virt_conc_Seif}.
\end{proof}

\begin{proposition} Let $F_0 \subset \Sigma_0 \times I$ be a Seifert surface of $K_0$. If a Seifert surface $F_1 \subset \Sigma_1 \times I$ of $K_1$ is concordant to $F_0$ and $F_1' \subset \Sigma \times I$ is another Seifert surface for $K_1$, then there is a Seifert surface $F_0' \subset \Sigma_0 \times I$ of $K_0$ that is concordant to $F_1'$ and can be obtained from $F_0$ by a sequence of isotopies, $S$-equivalences, and connected sums with parallel copies of $(\pm\Sigma) \times 1$.
\end{proposition}

\begin{proof} There is an integer $n$ such that $F_1'$ is of the form $F_1+n \cdot \Sigma_1$. By Theorem \ref{thm_ambient_Z}, $F_1+n \Sigma_1$ is virtually concordant to $F_0+n \cdot \Sigma_0$. Any Seifert surface of $K_0$ having this form proves the claim. \end{proof}

\section{Calculations $\&$ Examples} \label{sec_examples}

\subsection{Virtual Seifert surfaces} \label{sec_vss} Virtual Seifert surfaces \cite{vss} can be used to simplify calculations involving almost classical knots. First, a \emph{virtual disc-band surface} $F$ is a figure on $S^2$ consisting of a single $0$-handle $F^0$ and $1$-handles $F^1_1,\ldots,F_n^1$ attached to $\partial F^0$ so that the $1$-handles intersect themselves and each other in only classical crossings of bands or virtual crossings of bands (see Figure \ref{fig_virt_seif_defn}, center). A virtual disc-band surface may be converted to a disc-band surface on some $\Sigma \times I$ as follows. View the $2$-sphere on which $F$ is drawn as the boundary of a $3$-ball $B^3$. At each virtual crossing of band, attach a $3$-dimensional $1$-handle to $B^3$. The transition is depicted in Figure \ref{fig_virt_seif_defn}, right, where we again use the standard Heegard diagram notation for the $1$-handle addition. The added $1$-handle allows one band to pass over the other. 
\newline

\begin{figure}[htb]
\def\svgwidth{5.3in}
     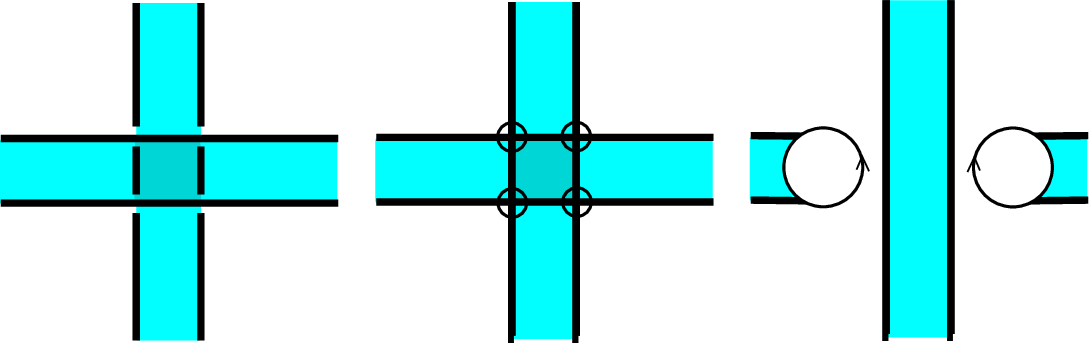
\caption{A classical crossing of bands (left)  and a virtual crossing of bands (center). Each virtual crossing of bands can be removed by adding a $1$-handle to obtain a disc-band surface on some surface $\Sigma$ (right).} 
\label{fig_virt_seif_defn}
\end{figure}

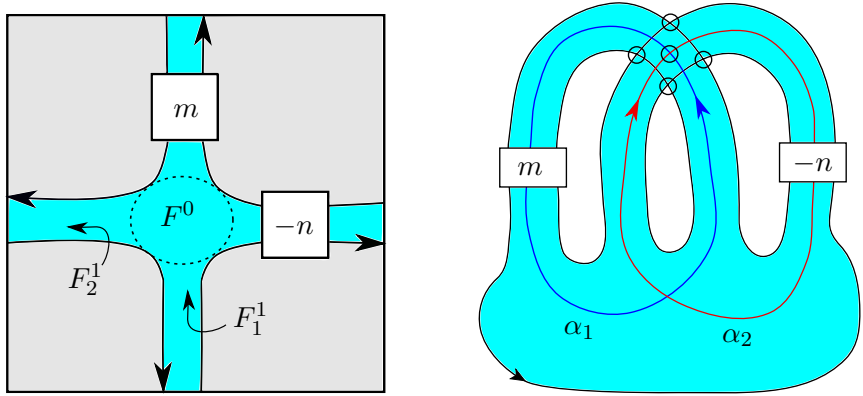
\begin{figure}[htb]
\begin{tabular}{ccc} & & \\ & & \\
 \def\svgwidth{2.00in}
\begingroup%
  \makeatletter%
  \providecommand\color[2][]{%
    \errmessage{(Inkscape) Color is used for the text in Inkscape, but the package 'color.sty' is not loaded}%
    \renewcommand\color[2][]{}%
  }%
  \providecommand\transparent[1]{%
    \errmessage{(Inkscape) Transparency is used (non-zero) for the text in Inkscape, but the package 'transparent.sty' is not loaded}%
    \renewcommand\transparent[1]{}%
  }%
  \providecommand\rotatebox[2]{#2}%
  \newcommand*\fsize{\dimexpr\f@size pt\relax}%
  \newcommand*\lineheight[1]{\fontsize{\fsize}{#1\fsize}\selectfont}%
  \ifx\svgwidth\undefined%
    \setlength{\unitlength}{236.51664303bp}%
    \ifx\svgscale\undefined%
      \relax%
    \else%
      \setlength{\unitlength}{\unitlength * \real{\svgscale}}%
    \fi%
  \else%
    \setlength{\unitlength}{\svgwidth}%
  \fi%
  \global\let\svgwidth\undefined%
  \global\let\svgscale\undefined%
  \makeatother%
  \begin{picture}(1,0.99743287)%
    \lineheight{1}%
    \setlength\tabcolsep{0pt}%
    \put(0,0){\includegraphics[width=\unitlength]{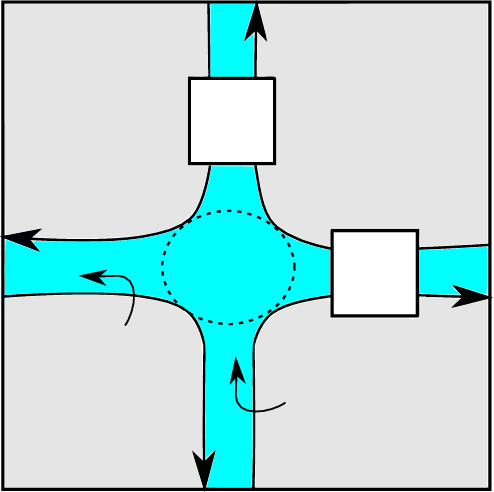}}%
    \put(0.70356767,0.42043317){\makebox(0,0)[lt]{\lineheight{1.25}\smash{\begin{tabular}[t]{l}$-n$\end{tabular}}}}%
    \put(0.441426,0.72688839){\makebox(0,0)[lt]{\lineheight{1.25}\smash{\begin{tabular}[t]{l}$m$\end{tabular}}}}%
    \put(0.40661632,0.44118612){\color[rgb]{0,0,0}\makebox(0,0)[lt]{\lineheight{40.54999924}\smash{\begin{tabular}[t]{l}$F^0$\end{tabular}}}}%
    \put(0.59714504,0.17599087){\color[rgb]{0,0,0}\makebox(0,0)[lt]{\lineheight{40.54999924}\smash{\begin{tabular}[t]{l}$F_1^1$\end{tabular}}}}%
    \put(0.16459341,0.27897924){\color[rgb]{0,0,0}\makebox(0,0)[lt]{\lineheight{40.54999924}\smash{\begin{tabular}[t]{l}$F_2^1$\end{tabular}}}}%
  \end{picture}%
\endgroup%
 & \quad \quad&
    \def\svgwidth{2.00in}
\begingroup%
  \makeatletter%
  \providecommand\color[2][]{%
    \errmessage{(Inkscape) Color is used for the text in Inkscape, but the package 'color.sty' is not loaded}%
    \renewcommand\color[2][]{}%
  }%
  \providecommand\transparent[1]{%
    \errmessage{(Inkscape) Transparency is used (non-zero) for the text in Inkscape, but the package 'transparent.sty' is not loaded}%
    \renewcommand\transparent[1]{}%
  }%
  \providecommand\rotatebox[2]{#2}%
  \newcommand*\fsize{\dimexpr\f@size pt\relax}%
  \newcommand*\lineheight[1]{\fontsize{\fsize}{#1\fsize}\selectfont}%
  \ifx\svgwidth\undefined%
    \setlength{\unitlength}{420.71801354bp}%
    \ifx\svgscale\undefined%
      \relax%
    \else%
      \setlength{\unitlength}{\unitlength * \real{\svgscale}}%
    \fi%
  \else%
    \setlength{\unitlength}{\svgwidth}%
  \fi%
  \global\let\svgwidth\undefined%
  \global\let\svgscale\undefined%
  \makeatother%
  \begin{picture}(1,1.02543137)%
    \lineheight{1}%
    \setlength\tabcolsep{0pt}%
    \put(0,0){\includegraphics[width=\unitlength]{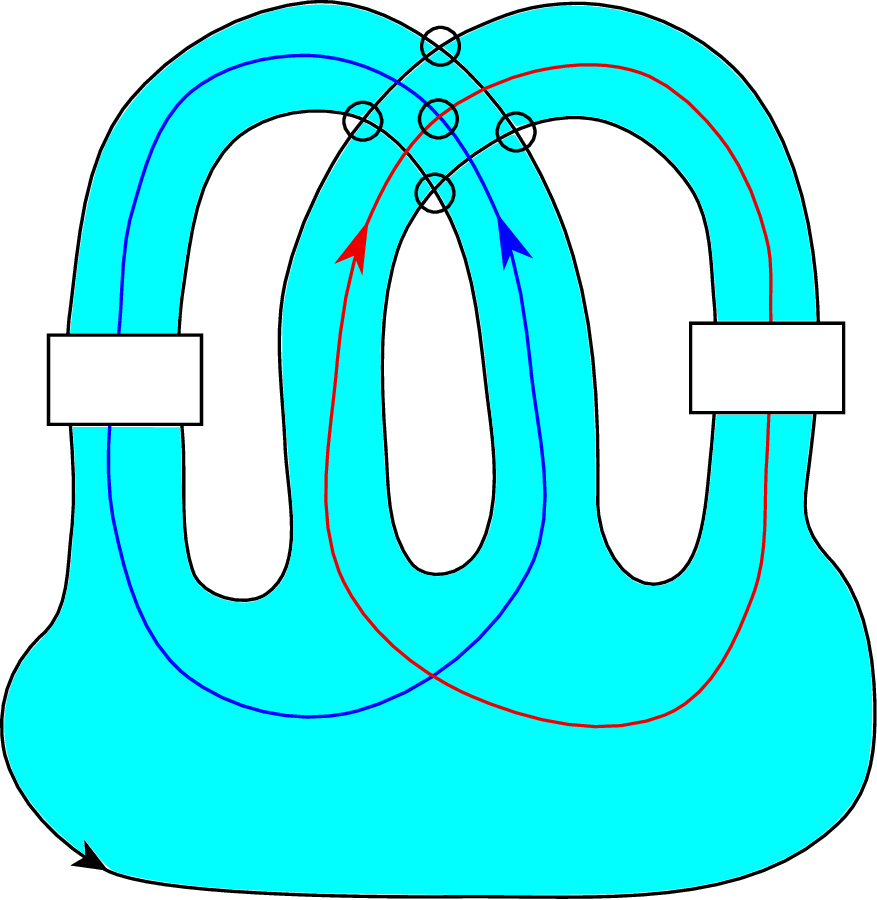}}%
    \put(0.22328315,0.15734331){\makebox(0,0)[lt]{\lineheight{1.25}\smash{\begin{tabular}[t]{l}$\alpha_1$\end{tabular}}}}%
    \put(0.63779181,0.13817061){\makebox(0,0)[lt]{\lineheight{1.25}\smash{\begin{tabular}[t]{l}$\alpha_2$\end{tabular}}}}%
    \put(0.10195729,0.57444542){\makebox(0,0)[lt]{\lineheight{1.25}\smash{\begin{tabular}[t]{l}$m$\end{tabular}}}}%
    \put(0.8228807,0.58506206){\makebox(0,0)[lt]{\lineheight{1.25}\smash{\begin{tabular}[t]{l}$-n$\end{tabular}}}}%
  \end{picture}%
\endgroup%

\end{tabular}
    \caption{The surface $F_0(m,n) \subset S^1 \times S^1 \times I$ (left) and its virtual disc-band surface (right).}
    \label{fig_kmn_disc_band}
\end{figure}

\begin{example} \label{example_F_0_third} The left side of Figure \ref{fig_kmn_disc_band} depicts a handle decomposition $F=F^0 \cup F_1^1 \cup F_2^1$ of a Seifert surface in $S^1 \times S^1 \times I$. The opposite sides of the large gray square are identified to make a torus $S^1 \times S^1$. The squares labeled $m$ and $-n$ indicate $m$ full positive twists of the $1$-handle $F_1^1$ and $n$ full negative twists of the $1$-handle $F_2^1$. The corresponding virtual disc-band surface is drawn in Figure \ref{fig_kmn_disc_band}. This figure is obtained by connecting the edges of the $1$-handles in $S^2$ according to the identifications of the torus. This forces the two bands to intersect transversely. Marking this transverse intersection as a virtual crossing of bands yields the depicted virtual disc-band surface. Relative to the ordered basis $\{\alpha_1,\alpha_2\}$, the directed Seifert matrices are exactly the matrices $A^+_0,A^-_0$ from Example \ref{example_mnmat_1}. Note that here we are using the non-standard ``left-hand rule'' for determining the direction of the push-offs (Boden et al. \cite{acpaper}). We have chosen to continue with this in order to be consistent with earlier papers (e.g. \cite{bcg2,vss}). \hfill $\square$ 
\end{example}

In \cite{vss}, it was shown that any Seifert surface in $\Sigma \times I$ can be drawn as a virtual disc-band surface on $S^2$. Moreover, given a Gauss diagram of an almost classical knot, the \emph{virtual Seifert surface algorithm} can be used to draw a figure on $S^2$ called a \emph{virtual Seifert surface} \cite{vss}. In this case, the classical crossings need not occur in band crossings but all virtual crossings must occur in virtual band crossings. The virtual knot bounded by a virtual Seifert surface is called its \emph{virtual boundary}. Using a virtual version of topological script, every virtual Seifert surface can be deformed to a virtual disc-band surface. The virtual Seifert surfaces used in the examples below (see Figures \ref{fig_5_2433} and \ref{fig_6_85091}) have been produced from the Gauss codes in Green's table \cite{green} using this method. Otherwise, the virtual Seifert surface algorithm is not used here. For further details, the reader is referred to \cite{vss}. In the next example, we show how virtual Seifert surfaces can be used to construct virtual concordances of Seifert surfaces.

\begin{example} \label{example_slice_movie_F_0} Here we will show that the Seifert surface $F$ in Figure \ref{fig_kmn_disc_band} is slice whenever $m=n$. Figure \ref{fig_k_1_1_movie} shows a virtual concordance to the disc for the case of $m=n=1$. The $t=0$ frame shows a virtual disc-band surface $F$. The frame $t=\tfrac{1}{4}$ shows the addition of a saddle. The dotted circle in this frame shows the boundary of compression disc. Compressing along this disc gives the virtual disc-band surface at $t=\tfrac{1}{2}$. After an isotopy, there is a single virtual band crossing. This ``virtual curl'' can be untwisted at the cost of adding a $2$-handle to the ambient space. To see this, temporarily convert back to the Heegard diagram notation. This is shown to the right of the $t=\tfrac{1}{2}$ frame. Attaching a $3$-dimensional $2$-handle along the magenta curve cancels the $1$-handle. For aid in visualization, the attaching sphere of the $2$-handle is also drawn as a virtual knot in the $t=\tfrac{1}{2}$ frame. Thus, we see that untwisting the virtual curl gives the annulus at $t=\tfrac{3}{4}$. Lastly, we contract the white disc in the $t=\tfrac{3}{4}$ frame to a point. This gives a disc at $t=1$. Thus $F$ is slice when $m=n=1$. The case of arbitrary $m=n$ follows similarly. \hfill $\square$
\end{example}

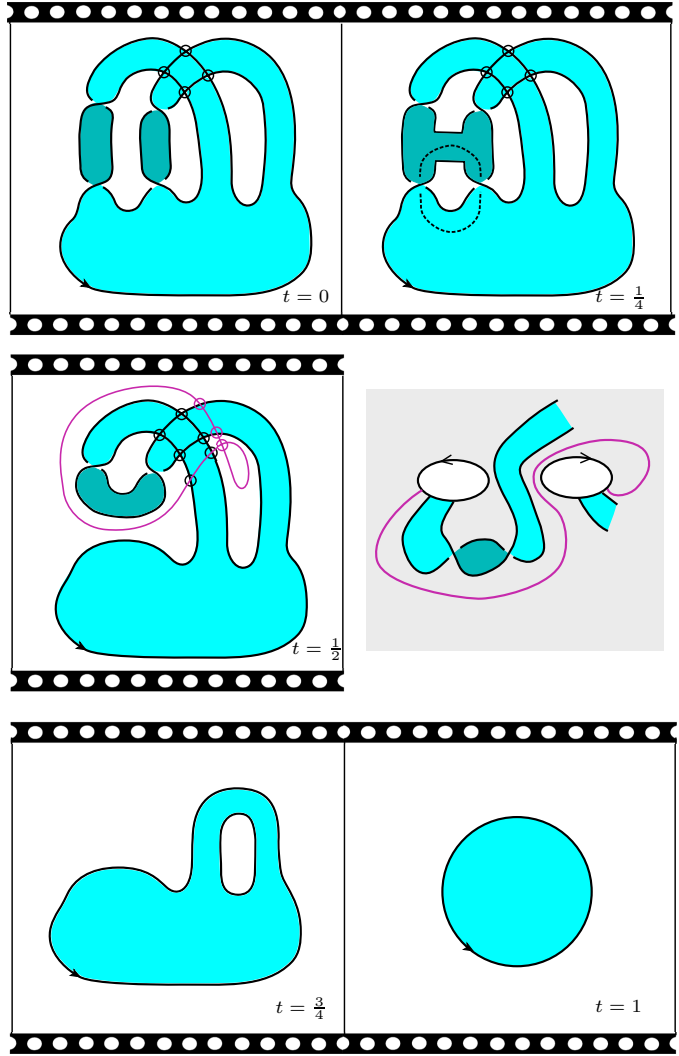
\begin{figure}[htb]
    \def\svgwidth{3.5in}\tiny
\begingroup%
  \makeatletter%
  \providecommand\color[2][]{%
    \errmessage{(Inkscape) Color is used for the text in Inkscape, but the package 'color.sty' is not loaded}%
    \renewcommand\color[2][]{}%
  }%
  \providecommand\transparent[1]{%
    \errmessage{(Inkscape) Transparency is used (non-zero) for the text in Inkscape, but the package 'transparent.sty' is not loaded}%
    \renewcommand\transparent[1]{}%
  }%
  \providecommand\rotatebox[2]{#2}%
  \newcommand*\fsize{\dimexpr\f@size pt\relax}%
  \newcommand*\lineheight[1]{\fontsize{\fsize}{#1\fsize}\selectfont}%
  \ifx\svgwidth\undefined%
    \setlength{\unitlength}{346.85063296bp}%
    \ifx\svgscale\undefined%
      \relax%
    \else%
      \setlength{\unitlength}{\unitlength * \real{\svgscale}}%
    \fi%
  \else%
    \setlength{\unitlength}{\svgwidth}%
  \fi%
  \global\let\svgwidth\undefined%
  \global\let\svgscale\undefined%
  \makeatother%
  \begin{picture}(1,1.57363842)%
    \lineheight{1}%
    \setlength\tabcolsep{0pt}%
    \put(0,0){\includegraphics[width=\unitlength]{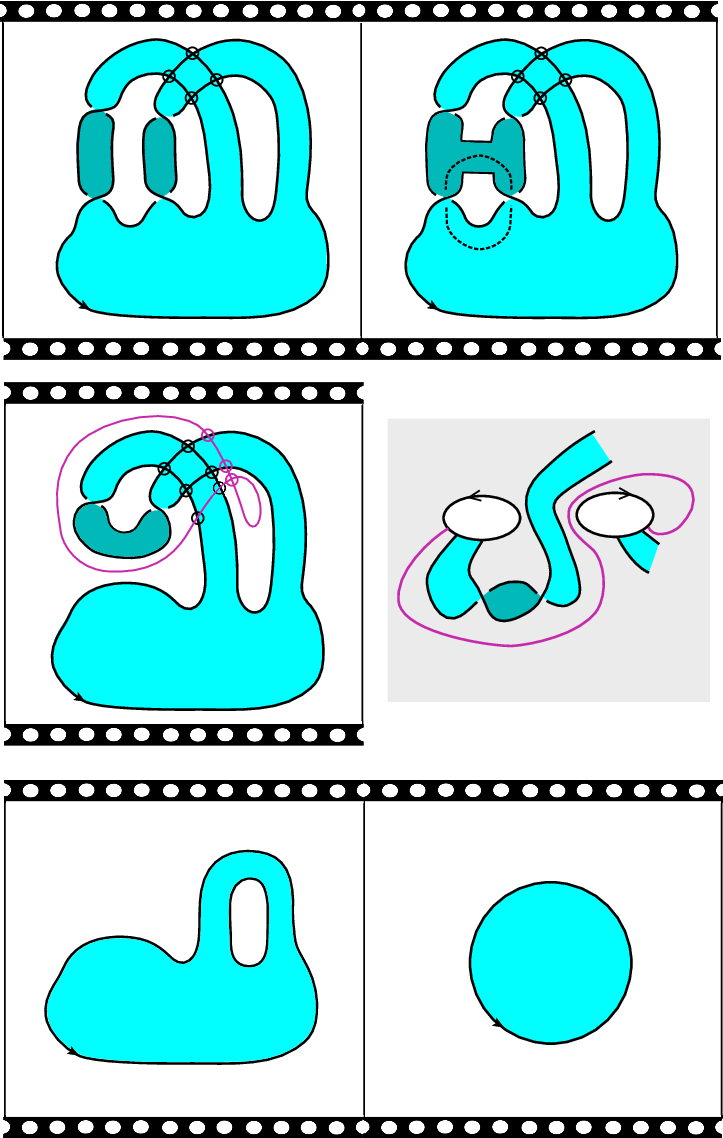}}%
    \put(0.41093062,1.12398569){\color[rgb]{0,0,0}\makebox(0,0)[lt]{\lineheight{40.54999924}\smash{\begin{tabular}[t]{l}$t=0$\end{tabular}}}}%
    \put(0.87857576,1.12398569){\color[rgb]{0,0,0}\makebox(0,0)[lt]{\lineheight{40.54999924}\smash{\begin{tabular}[t]{l}$t=\tfrac{1}{4}$\end{tabular}}}}%
    \put(0.4246359,0.59846228){\color[rgb]{0,0,0}\makebox(0,0)[lt]{\lineheight{40.54999924}\smash{\begin{tabular}[t]{l}$t=\tfrac{1}{2}$\end{tabular}}}}%
    \put(0.40061795,0.06045942){\color[rgb]{0,0,0}\makebox(0,0)[lt]{\lineheight{40.54999924}\smash{\begin{tabular}[t]{l}$t=\tfrac{3}{4}$\end{tabular}}}}%
    \put(0.87857576,0.0628612){\color[rgb]{0,0,0}\makebox(0,0)[lt]{\lineheight{40.54999924}\smash{\begin{tabular}[t]{l}$t=1$\end{tabular}}}}%
  \end{picture}%
\endgroup%
 \normalsize
    \caption{A virtual concordance movie using virtual Seifert surfaces. See Example \ref{example_slice_movie_F_0}.}
    \label{fig_k_1_1_movie}
\end{figure}

\subsection{Arf invariant example} \label{sec_arf_example} Let $K$ be the virtual knot $5.2433$ from Green's table \cite{green}. It is almost classical (see \cite{bcg2,acpaper}). The virtual disc-band surface $F$ drawn in Figure \ref{fig_5_2433} can be obtained using the virtual Seifert surface algorithm \cite{vss}. Relative to the given basis, the directed Seifert matrices are:
\[
A^+=\left[
\begin{array}{rrrr}
 1 & 0 & -1 & 0 \\
 0 & 1 & 0 & -1 \\
 0 & -1 & 1 & 0 \\
 0 & 0 & 0 & 1 \\
\end{array}
\right], \quad A^-=\left[
\begin{array}{rrrr}
 1 & 1 & -1 & -1 \\
 -1 & 1 & 0 & -1 \\
 0 & -1 & 1 & 1 \\
 1 & 0 & -1 & 1 \\
\end{array}
\right]
\]
Set $\textbf{A}=(A^+,A^-)$. The pairing $\star_F$ is non-singular, as can be seen by calculating $q_{\textbf{A}}(x+y)-q_{\textbf{A}}(x)-q_{\textbf{A}}(y)$ for every pair of elements $x,y$ in the given basis for $H_1(F;\mathbb{F}_2)$. The result is the non-singular matrix:
\[
\left[
\begin{array}{cccc}
 0 & 0 & 1 & 0 \\
 0 & 0 & 1 & 1 \\
 1 & 1 & 0 & 0 \\
 0 & 1 & 0 & 0 \\
\end{array}.
\right]
\]
Computing $q_{\textbf{A}}(x)$ for all $x \in H_1(F;\mathbb{F}_2)$, we see that $\text{Arf}(q_{\textbf{A}})=1$.

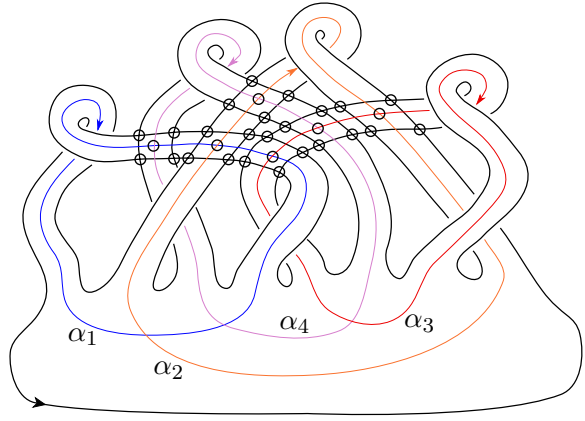
\begin{figure}[htb]
    \centering
    \def\svgwidth{3in}
\begingroup%
  \makeatletter%
  \providecommand\color[2][]{%
    \errmessage{(Inkscape) Color is used for the text in Inkscape, but the package 'color.sty' is not loaded}%
    \renewcommand\color[2][]{}%
  }%
  \providecommand\transparent[1]{%
    \errmessage{(Inkscape) Transparency is used (non-zero) for the text in Inkscape, but the package 'transparent.sty' is not loaded}%
    \renewcommand\transparent[1]{}%
  }%
  \providecommand\rotatebox[2]{#2}%
  \newcommand*\fsize{\dimexpr\f@size pt\relax}%
  \newcommand*\lineheight[1]{\fontsize{\fsize}{#1\fsize}\selectfont}%
  \ifx\svgwidth\undefined%
    \setlength{\unitlength}{681.91206075bp}%
    \ifx\svgscale\undefined%
      \relax%
    \else%
      \setlength{\unitlength}{\unitlength * \real{\svgscale}}%
    \fi%
  \else%
    \setlength{\unitlength}{\svgwidth}%
  \fi%
  \global\let\svgwidth\undefined%
  \global\let\svgscale\undefined%
  \makeatother%
  \begin{picture}(1,0.72073851)%
    \lineheight{1}%
    \setlength\tabcolsep{0pt}%
    \put(0,0){\includegraphics[width=\unitlength]{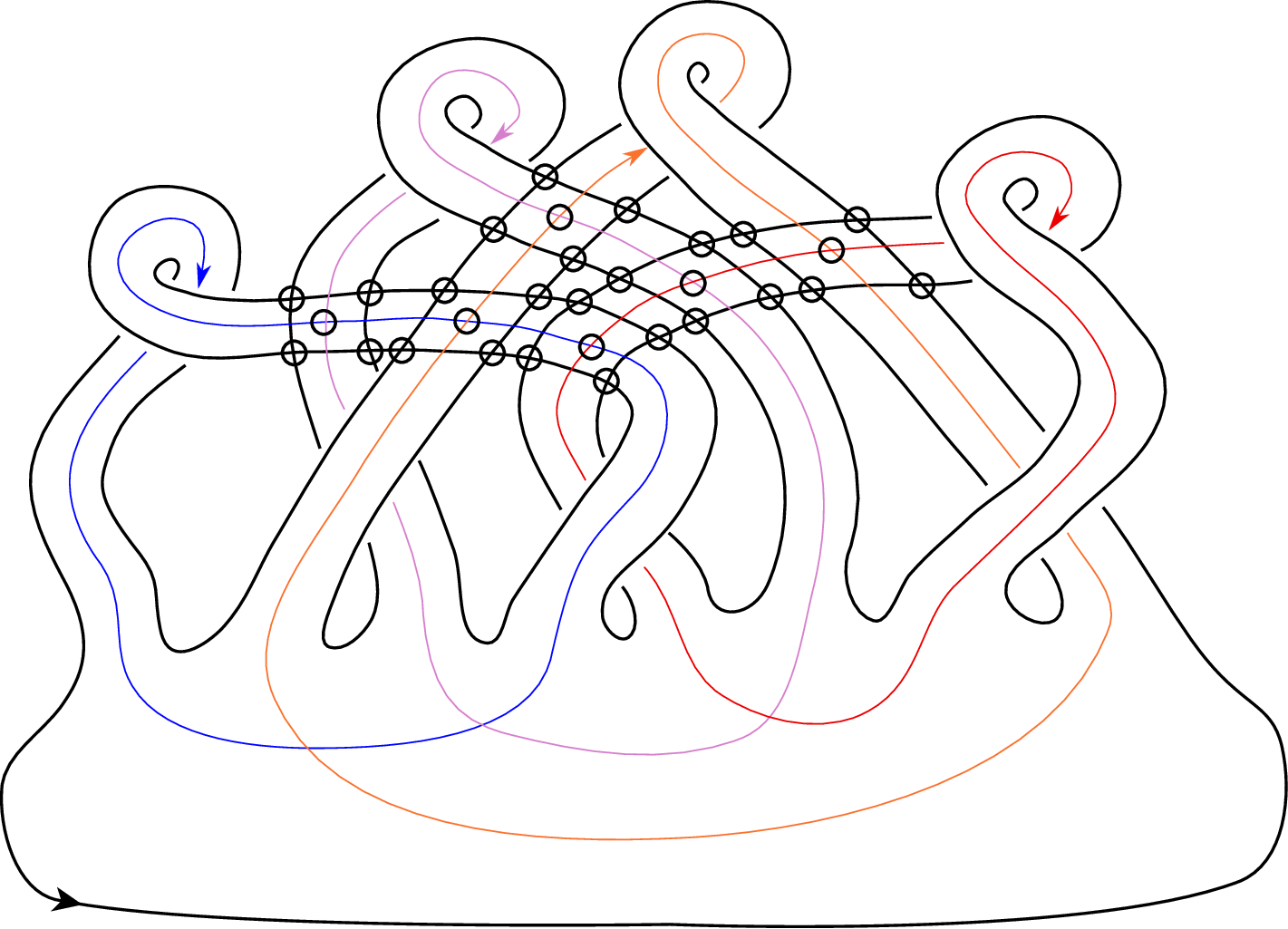}}%
    \put(0.10194479,0.13094637){\makebox(0,0)[lt]{\lineheight{1.25}\smash{\begin{tabular}[t]{l}$\alpha_1$\end{tabular}}}}%
    \put(0.25088107,0.07247254){\makebox(0,0)[lt]{\lineheight{1.25}\smash{\begin{tabular}[t]{l}$\alpha_2$\end{tabular}}}}%
    \put(0.68891429,0.15411529){\makebox(0,0)[lt]{\lineheight{1.25}\smash{\begin{tabular}[t]{l}$\alpha_3$\end{tabular}}}}%
    \put(0.47124939,0.15184978){\makebox(0,0)[lt]{\lineheight{1.25}\smash{\begin{tabular}[t]{l}$\alpha_4$\end{tabular}}}}%
  \end{picture}%
\endgroup%

    \caption{A virtual disc-band surface for the almost classical knot 5.2433.}
    \label{fig_5_2433}
\end{figure}

\subsection{Guide to calculating algebraic concordance order} \label{sec_calculating} By Theorem \ref{thm_classify_VG}, $\mathscr{VG}^{\mathbb{Q}} \cong \mathscr{I}(\mathbb{Q})\oplus \mathscr{G}^{\mathbb{Q}}$. Hence we can determine the order in $\mathscr{VG}^{\mathbb{Q}}$ by determining the order in each summand. For the reader's convenience, we recall some facts about calculating torsion in $\mathscr{W}(\mathbb{Q})$ and $\mathscr{G}^{\mathbb{Q}}$. The Witt group $\mathscr{W}(\mathbb{Q})$ has a canonical decomposition in terms of $\mathscr{W}(\mathbb{R})$ and $\mathscr{W}(\mathbb{F}_p)$ (Scharlau \cite{scharlau}, Theorem 5.3.4). The isomorphism is: 
\[
\mathscr{W}(\mathbb{Q}) \cong \mathscr{W}(\mathbb{R}) \oplus \coprod_{p} \mathscr{W}(\mathbb{Z}/p\mathbb{Z}),
\]
 where the coproduct is taken over all primes $p \in \mathbb{N}$. Real forms in $\mathscr{W}(\mathbb{R}) \cong \mathbb{Z}$ are classified by their signature. For $\mathbb{Z}/p\mathbb{Z}$, canonical maps $\partial_p :\mathscr{W}(\mathbb{Q}) \to \mathscr{W}(\mathbb{Z}/2\mathbb{Z})$ are defined as follows. First, diagonalize a given rational form $q$ over $\mathbb{Q}$ as $\langle a_1 \rangle \perp \cdots \perp \langle a_m \rangle$. For each $i$, write $a_i=a_i'p^{k_i}$, where $a_i$ is relatively prime to $p$. If $k_i$ is even, set $\partial_p \langle a_i \rangle$ to be the trivial form in $\mathscr{W}(\mathbb{Z}/p\mathbb{Z})$. If $k$ is odd, set $\partial_p \langle a_i \rangle=\langle [a_i'] \rangle$. Then $\partial_p(q)=\bigperp_{k_i \text{ odd}} \langle [a_i'] \rangle$. Note that $\partial_p(q)$ is trivial for all but finitely many $p$.
\newline
\newline
For $p=2$, $\mathscr{W}(\mathbb{Z}/p\mathbb{Z}) \cong \mathbb{Z}/2\mathbb{Z}$. The isomorphism is the rank modulo $2$ (i.e. the dimension index). For $p \equiv 1 \pmod{4}$, $\mathscr{W}(\mathbb{Z}/p\mathbb{Z}) \cong  \mathbb{Z}/2\mathbb{Z} \oplus \mathbb{Z}/2\mathbb{Z}$. The isomorphism is given by $(e,d):\mathscr{W}(\mathbb{Z}/p\mathbb{Z}) \to \mathbb{Z}/2\mathbb{Z} \oplus \mathbb{Z}/2\mathbb{Z}$ where $e$ is the dimension index and $d$ is the discriminant. For $p \equiv 3 \pmod{4}$, $\mathscr{W}(\mathbb{Z}/p\mathbb{Z}) \cong \mathbb{Z}/4\mathbb{Z}$. In this case, odd rank forms have have order $4$.
\newline
\newline
Here we will use four well-known facts to calculate concordance order in $\mathscr{G}^{\mathbb{Q}} \cong \mathscr{G}_{\mathbb{Q}}$. Proofs and further discussion can be found in the indicated references.

\begin{fact} [see Livingston \cite{liv_aco}, Theorem 2.2] \label{fact_livingston_sig} Let $A$ be an integral Seifert matrix of a classical knot $K$. Then $A \in \mathscr{G}^{\mathbb{Q}}$ has finite order if and only if it has vanishing signature function.
\end{fact}

\begin{fact}[see Livingston \cite{liv_aco}, Corollary 3.3] \label{fact_livingston_not_4} Suppose $A \in \mathscr{G}^{\mathbb{Q}}$ arises from the Seifert matrix of a classical knot. Suppose that the Alexander polynomial $\Delta_A(t)=\det(A-t A^{\intercal})$ has the property that $\Delta_A(-1)$ is not divisible by a prime $p$ satisfying $p \equiv 3 \pmod 4$. Then $A$ does not have order $4$.
\end{fact}

\begin{fact} [see Levine \cite{levine_2}, Corollary 23a,c] \label{fact_levine} Let $\beta=(V,B,S) \in \mathscr{G}_{\mathbb{Q}}$ be an isometric structure such that $\Delta_S(t)=\prod_i (\lambda_i(t))^{e_i}$, where each $\lambda_i(t)$  is an irreducible quadratic. If $\lambda_i(1)\lambda_i(-1)<0$ for all $i$, then $\beta$ has finite order. If $\beta$ has finite order, $\beta$ has order $4$ if and only if there is a prime $p \equiv 3 \pmod{4}$ and a symmetric factor $\lambda_i(t)$ such that $e_i$ is odd and $\lambda_i(1)\lambda_i(-1)=p^aq$, where $a$ is odd and $\gcd(p,q)=1$.
\end{fact}

\begin{fact}[see Livingston \cite{liv_aco}, Section 4.1] \label{fact_livingston_order_2} If $(V,B,S) \in \mathscr{G}_{\mathbb{Q}}$ is of finite order $\ne 4$ and $\Delta_S(t)$ has an irreducible symmetric factor $\lambda(t)$ having odd exponent in $\Delta_S(t)$, then $(V,B,S)$ has order $2$. 
\end{fact}

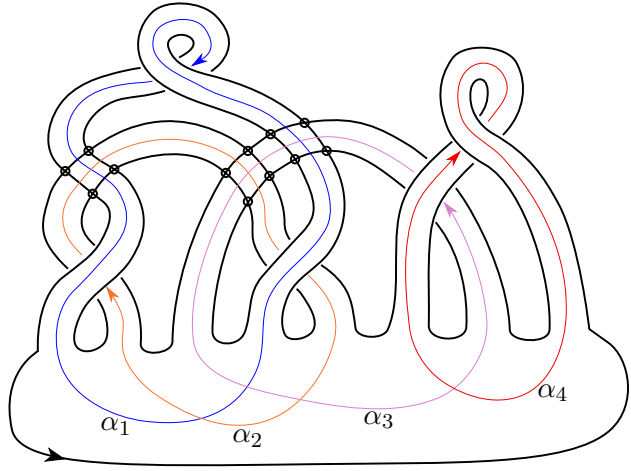
\begin{figure}[htb]
\def\svgwidth{3.25in}
\begingroup%
  \makeatletter%
  \providecommand\color[2][]{%
    \errmessage{(Inkscape) Color is used for the text in Inkscape, but the package 'color.sty' is not loaded}%
    \renewcommand\color[2][]{}%
  }%
  \providecommand\transparent[1]{%
    \errmessage{(Inkscape) Transparency is used (non-zero) for the text in Inkscape, but the package 'transparent.sty' is not loaded}%
    \renewcommand\transparent[1]{}%
  }%
  \providecommand\rotatebox[2]{#2}%
  \newcommand*\fsize{\dimexpr\f@size pt\relax}%
  \newcommand*\lineheight[1]{\fontsize{\fsize}{#1\fsize}\selectfont}%
  \ifx\svgwidth\undefined%
    \setlength{\unitlength}{477.64103766bp}%
    \ifx\svgscale\undefined%
      \relax%
    \else%
      \setlength{\unitlength}{\unitlength * \real{\svgscale}}%
    \fi%
  \else%
    \setlength{\unitlength}{\svgwidth}%
  \fi%
  \global\let\svgwidth\undefined%
  \global\let\svgscale\undefined%
  \makeatother%
  \begin{picture}(1,0.74807515)%
    \lineheight{1}%
    \setlength\tabcolsep{0pt}%
    \put(0,0){\includegraphics[width=\unitlength]{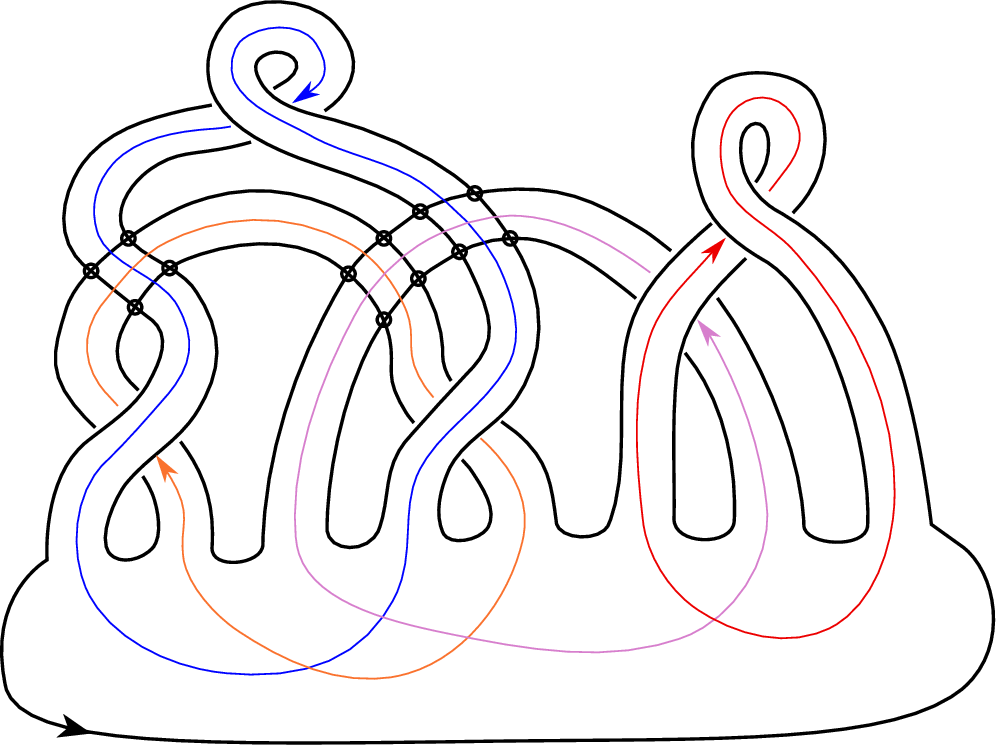}}%
    \put(0.14826193,0.05956761){\makebox(0,0)[lt]{\lineheight{1.25}\smash{\begin{tabular}[t]{l}$\alpha_1$\end{tabular}}}}%
    \put(0.36038065,0.04){\makebox(0,0)[lt]{\lineheight{1.25}\smash{\begin{tabular}[t]{l}$\alpha_2$\end{tabular}}}}%
    \put(0.57194717,0.07){\makebox(0,0)[lt]{\lineheight{1.25}\smash{\begin{tabular}[t]{l}$\alpha_3$\end{tabular}}}}%
    \put(0.85179228,0.11324123){\makebox(0,0)[lt]{\lineheight{1.25}\smash{\begin{tabular}[t]{l}$\alpha_4$\end{tabular}}}}%
  \end{picture}%
\endgroup%

\caption{A virtual disc-band surface for the almost classical knot 6.85091.} \label{fig_6_85091}
\end{figure}

\subsection{Non-classical torsion in algebraic concordance} \label{sec_nonclass_torsion} A virtual disc-band surface $F$ of the almost classical knot 6.85091 is shown in Figure \ref{fig_6_85091}. We will show that this Seifert surface is not concordant to a classical Seifert surface but has finite order directed Seifert matrices in $\mathscr{VG}^{\mathbb{Q}}$. Relative to the indicated basis of $H_1(F;\mathbb{Z})$, it follows that:
\[
A^+=
\begin{bmatrix*}[r]
 1 & -1 & -1 & 0 \\
 -1 & 0 & 1 & 0 \\
 1 & -1 & 0 & 1 \\
 0 & 0 & 0 & 1 \\
\end{bmatrix*}
, \quad\quad A^-=
\begin{bmatrix*}[r]
 1 & -2 & 0 & 0 \\
 0 & 0 & 0 & 0 \\
 0 & 0 & 0 & 0 \\
 0 & 0 & 1 & 1 \\
\end{bmatrix*}
\]
That $A^-$ is metabolic in $\mathscr{VG}^{\mathbb{Q}}$ is clear. To show that $A^+$ has order $2$, we use the isomorphism $\mathscr{VG}^{\mathbb{Q}} \cong \mathscr{VG}_{\mathbb{Q}}$. Set $V=H_1(F;\mathbb{Q})\cong \mathbb{Q}^4$. The directed isometric structure $\beta=(V,B,S)$ is given by:
\[
B=A^++(A^+)^{\intercal}=
\begin{bmatrix*}[r]
 2 & -2 & 0 & 0 \\
 -2 & 0 & 0 & 0 \\
 0 & 0 & 0 & 1 \\
 0 & 0 & 1 & 2 \\
\end{bmatrix*},  \quad\quad S=(A^+)^{-1}(A^+)^{\intercal}=
\begin{bmatrix*}[r]
 -1 & 2 & -1 & -1 \\
 0 & 1 & 0 & 0 \\
 -2 & 2 & -2 & -1 \\
 0 & 0 & 1 & 1 \\
\end{bmatrix*}
\]
The characteristic polynomial of $S$ is $ \Delta_S(t)=(t-1)^2 \left(t^2+3 t+1\right)$. The primary components are: 
\begin{align*}
V_{t-1} &= \text{span} \left(\{(1, 1, 0, 0)^{\intercal}, (-1, 0, 0, 2)^{\intercal}\} \right)\\
V_{t^2+3t+1} &= \text{span}\left(\{(1, 0, 1, 0)^{\intercal}, (2, 0, 0, 1)^{\intercal}\} \right)
\end{align*}
Let $B'$ and $S'$ denote the restrictions of $B$ and $S$ to $V_{t-1}$, respectively. Likewise, let $B'',S''$ denote the restrictions of $B,S$ to $V_{t^2+3t+1}$. Then $\beta'=(V_{t-1},B',S') \in \mathscr{I}(\mathbb{Q})$, $\beta''=(V_{t^2+3t+1},B'',S'') \in \mathscr{G}_{\mathbb{Q}}$ and we have decomposed $(V,B,S)$ as an element in $\mathscr{I}(\mathbb{Q}) \oplus \mathscr{G}_{\mathbb{Q}}$. The block decomposition is:
\[
B' \oplus B''= \left[
\begin{array}{rr|rr}
 -2 & 0 & 0 & 0 \\
 0 & 10 & 0 & 0 \\ \hline
 0 & 0 & 2 & 5 \\
 0 & 0 & 5 & 10 \\
\end{array}
\right], \quad\quad S' \oplus S''=\left[
\begin{array}{rr|rr}
 1 & 0 & 0 & 0 \\
 0 & 1 & 0 & 0 \\ \hline
 0 & 0 & -4 & -5 \\
 0 & 0 & 1 & 1 \\
\end{array}
\right]
\]
We now show that both $\beta'$ and $\beta''$ have order $2$, so that $\beta$ also has order $2$. Consider first the case of $\beta''$. Since $\Delta_{S''}(t)=t^2+3t+1$, $\Delta_{S''}(t)$ is a product of symmetric irreducible quadratics over $\mathbb{Q}$. Hence, the order may be computed using Fact \ref{fact_levine}. That $\beta''$  has finite order follows from $\Delta_{S''}(-1)\Delta_{S''}(1)=-5<0$. Since the only prime dividing $\Delta_{S''}(-1)\Delta_{S''}(1)$ is $5$, it follows that $\beta''$ cannot have order $4$. By Fact \ref{fact_livingston_order_2}, $\beta''$ has order $2$ in $\mathscr{G}_{\mathbb{Q}}$.
\newline
\newline
Now consider the case of $\beta'$. The corresponding quadratic form is $\langle -2\rangle \perp \langle 10 \rangle$. Since this has zero signature, $\langle -2\rangle \perp \langle 10 \rangle$ has finite order. To determine the order, we must compute $\partial_p(\langle -2\rangle \perp \langle 10 \rangle)$ for all $p=2,3,5,7,\ldots$. For primes $p \ne 5$, $\partial_p(\langle -2\rangle \perp \langle 10 \rangle)$ is the trivial class in $\mathscr{W}(\mathbb{F}_p)$.  For $p=5$, $\partial_5(\langle -2\rangle \perp \langle 10 \rangle)=\langle 2 \rangle$. Since $p \equiv 1 \pmod{4}$, $\mathscr{W}(\mathbb{F}_5) \cong \mathbb{Z}/2\mathbb{Z} \oplus \mathbb{Z}/2\mathbb{Z}$. This is an odd rank form and hence cannot be trivial. Thus $\beta'$ has order $2$ in $\mathscr{I}(\mathbb{Q})$ and hence $A^+$ has order $2$ in $\mathscr{VG}^{\mathbb{Q}}$.
\newline
\newline
Since $A^+$ and $A^-$ have different orders in $\mathscr{VG}^{\mathbb{Q}}$, it follows from Proposition \ref{prop_g_in_eqlzr}, that $(A^+,A^-)$ cannot be algebraically concordant to a Seifert couple of a classical knot in $S^3$. It is also known that $6.85091$ is not concordant to a classical knot, since it has graded genus $1$ \cite{bcg1}. Thus, 6.85091 is an almost classical knot that is not concordant to a classical knot, admits a Seifert couple $(A^+,A^-)$ that is not algebraically concordant to a classical Seifert couple, and $A^+$ is of order $2$ in $\mathscr{VG}^{\mathbb{Q}}$. According to the symmetry data in Green's table \cite{green}, 6.85091 is equivalent to both its mirror images and their inverses. This suggests that there is a choice of base point for 6.85091 such that the resulting long virtual knot is not concordant to any classical knot but has finite order in the long virtual knot concordance group. We state this as a conjecture.

\begin{conjecture} The almost classical knot 6.85091 is the closure of a long virtual knot $\upsilon$ that has order 2 in $\mathscr{VC}$ and is not concordant to a classical knot.
\end{conjecture}

\begin{figure}[htb]
\def\svgwidth{4in}
\begingroup%
  \makeatletter%
  \providecommand\color[2][]{%
    \errmessage{(Inkscape) Color is used for the text in Inkscape, but the package 'color.sty' is not loaded}%
    \renewcommand\color[2][]{}%
  }%
  \providecommand\transparent[1]{%
    \errmessage{(Inkscape) Transparency is used (non-zero) for the text in Inkscape, but the package 'transparent.sty' is not loaded}%
    \renewcommand\transparent[1]{}%
  }%
  \providecommand\rotatebox[2]{#2}%
  \newcommand*\fsize{\dimexpr\f@size pt\relax}%
  \newcommand*\lineheight[1]{\fontsize{\fsize}{#1\fsize}\selectfont}%
  \ifx\svgwidth\undefined%
    \setlength{\unitlength}{345.61046119bp}%
    \ifx\svgscale\undefined%
      \relax%
    \else%
      \setlength{\unitlength}{\unitlength * \real{\svgscale}}%
    \fi%
  \else%
    \setlength{\unitlength}{\svgwidth}%
  \fi%
  \global\let\svgwidth\undefined%
  \global\let\svgscale\undefined%
  \makeatother%
  \begin{picture}(1,0.55882715)%
    \lineheight{1}%
    \setlength\tabcolsep{0pt}%
    \put(0,0){\includegraphics[width=\unitlength]{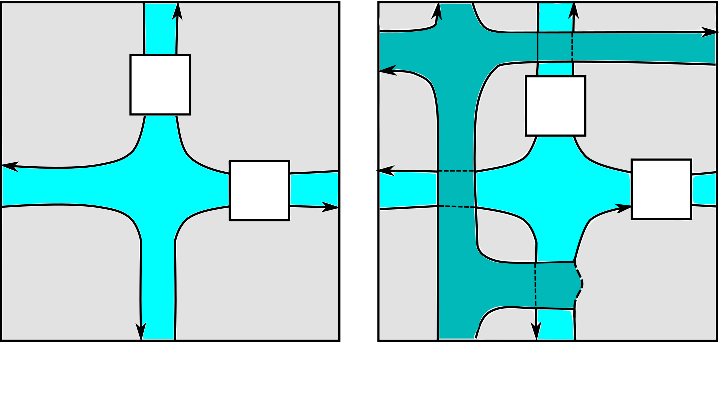}}%
    \put(0.88265846,0.28485811){\makebox(0,0)[lt]{\lineheight{1.25}\smash{\begin{tabular}[t]{l}$-n$\end{tabular}}}}%
    \put(0.13454444,0.01388029){\makebox(0,0)[lt]{\lineheight{1.25}\smash{\begin{tabular}[t]{l}$F_0(m,n)$\end{tabular}}}}%
    \put(0.6485156,0.00948213){\makebox(0,0)[lt]{\lineheight{1.25}\smash{\begin{tabular}[t]{l}$F_{-1}(m,n)$\end{tabular}}}}%
    \put(0.20179591,0.43448295){\makebox(0,0)[lt]{\lineheight{1.25}\smash{\begin{tabular}[t]{l}$m$\end{tabular}}}}%
    \put(0.32594014,0.28302513){\makebox(0,0)[lt]{\lineheight{1.25}\smash{\begin{tabular}[t]{l}$-n$\end{tabular}}}}%
    \put(0.75435679,0.4056527){\makebox(0,0)[lt]{\lineheight{1.25}\smash{\begin{tabular}[t]{l}$m$\end{tabular}}}}%
  \end{picture}%
\endgroup%

\caption{Two surfaces of a knot $K(m,n) \subset S^1 \times S^1 \times I$, where $m,n \in \mathbb{N}$ represent $m$ full positive twists and $n$ full negative twists, respectively.} \label{fig_kmn_two_surf}
\end{figure}
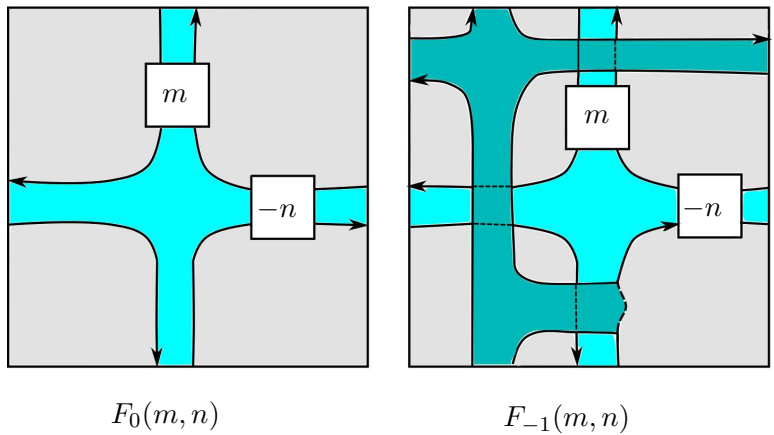

\subsection{Proof of Theorem \ref{thm_C}} \label{sec_thm_C_proof} Finally, we prove that there are infinitely many almost classical knots having Seifert matrices of all possible finite torsion in algebraic concordance. For this we will use the family of knots $K(m,n)$ shown in Figure \ref{fig_kmn_two_surf}. The squares labeled $m$ and $-n$ again correspond to $m$ positive full twists and $n$ negative full twists (see Figure \ref{fig_kmn_disc_band}).  A Seifert surface $F_0=F_0(m,n)$ is shown in Figure \ref{fig_kmn_two_surf}, left. The Seifert matrices of $F_0(m,n)$ were previously calculated in Example \ref{example_F_0_third}. This Seifert couple will be denoted by $\textbf{A}_0$. The right of Figure \ref{fig_kmn_two_surf} shows a Seifert surface $F_{-1}=F_{-1}(m,n)$. Note that $F_{-1}$ is $S$-equivalent to any Seifert surface of the form $F_0(m,n)-\Sigma$ (see Section \ref{sec_ambient_Z}). The basis $\{\alpha_1,\alpha_2\}$ for $H_1(F_0;\mathbb{Z})$ can be extended to a basis for $H_1(F_{-1};\mathbb{Z})$ by adding generators $\alpha_3,\alpha_4$ on $F_{-1}$ that are parallel to $\alpha_1, \alpha_2$, respectively. Let $\textbf{A}_{-1}$ denote a $\mathbb{Z}$-Seifert couple for $F_{-1}(m,n)$ relative to $\{\alpha_1,\alpha_2,\alpha_3,\alpha_4\}$. Then $\textbf{A}_0$ and $\textbf{A}_{-1}$ are given by: 
\[
A_0^+=\begin{bmatrix}
m & 0 \\ 0 & -n
\end{bmatrix},
A_0^-=\begin{bmatrix}
m & 1 \\ -1 & -n
\end{bmatrix},A_{-1}^{+}=\left[\begin{array}{c|c} A_0^+ & 0 \\ \hline H_2 & H_2 \end{array} \right],A_{-1}^{-}=\left[\begin{array}{c|c} A_0^- & 0 \\ \hline H_2 & 0 \end{array} \right], \text{ where } H_2=\begin{bmatrix}
0 & 1 \\ -1 & 0
\end{bmatrix}.
\]
Observe that $A_{-1}^-$ is metabolic and hence corresponds to an element of order $1$ in $\mathscr{VG}^{\mathbb{Q}}$. Theorem \ref{thm_C} will then follow from the fact that there are infinitely many pairs $m,n$ such that $A_0^+$ has order $4$ and $A_{-1}^+$ has order $2$. The directed isometric structures for $A_0^+,A_{-1}^+$ will be denoted by $(\mathbb{Q}^2,B_0,S_0^+)$, $(\mathbb{Q}^4,B_{-1},S_{-1}^+)$, respectively.  The characteristic polynomials for $S_0^+$ and $S_{-1}^+$ are given below: 
\[
\Delta_{S_0^+}(t) = (t-1)^2, \quad   \Delta_{S_{-1}^+}(t) = t^4-\left(2+\frac{1}{m n}\right)t^2+1.
\]

\begin{lemma} \label{lemma_A_0} For distinct odd primes $m,n$, if $m \text{ or } n \equiv 3 \pmod{4}$, then $A_0^+$ has order $4$.
\end{lemma}

\begin{proof} Since $\mathscr{VG}^{\mathbb{Q}} \cong \mathscr{VG}_{\mathbb{Q}} \cong \mathscr{I}(\mathbb{Q}) \oplus \mathscr{G}_{\mathbb{Q}}$, it is sufficient to show that each component in the decomposition  $\mathscr{I}(\mathbb{Q}) \oplus \mathscr{G}_{\mathbb{Q}}$ has finite order.  Since $\Delta_{S_0^+}(t) = (t-1)^2$, the restriction of $B_0$ to the $(t-1)$-primary component of $(\mathbb{Q}^2,B_0,S_0^{+})$ is $B_0$. Thus, $(\mathbb{Q}^2,B_0,S_0^{+})$ lies in $\mathscr{I}(\mathbb{Q})$. Since $B_0=A_0^++(A_0^+)^{\intercal}$ has signature $0$, this isometric structure has finite order in $\mathscr{VG}_{\mathbb{Q}}$.
\newline
\newline
The order of $A_0^+$ may therefore be determined using the classifying invariants of $\mathscr{W}(\mathbb{Q})$ (see Section \ref{sec_calculating}). Note that $\partial_m\langle 2m,-2n \rangle=\langle 2 \rangle$ is an odd rank form in $\mathscr{W}(\mathbb{F}_m)$. Similarly, $\partial_n\langle 2m,-2n \rangle=\langle -2 \rangle$ is an odd rank form in $\mathscr{W}(\mathbb{F}_n)$. If one of $m$ or $n$ is $3 \pmod{4}$, then at least one of $\mathscr{W}(\mathbb{F}_m)$ and $\mathscr{W}(\mathbb{F}_n)$ has an element of order $4$. As these correspond to the forms of odd rank, either $\partial_m\langle 2m,-2n \rangle$ or $\partial_m\langle 2m,-2n \rangle$ has order $4$. \end{proof}

\begin{lemma}\label{lemma_A_minus} If $mn$ and $4mn+1$ are not rational squares, then $A_{-1}^+$ has finite order $2$.
\end{lemma}
\begin{proof} We first show that $A^+_{-1}$ is of finite order. First note that $1$ is not a root of $\Delta_{S^+_{-1}}(t)$. This implies that the directed isometric structure lies in the $\mathscr{G}_{\mathbb{Q}}$ summand of $\mathscr{VG}_{\mathbb{Q}}$. Furthermore, we have that $\Delta_{K,F_{-1}}^+(t)=-m n t^4+(2 m n +1)t^2-m n$. Then $\Delta_{K,F_{-1}}^+(1)=1$. This implies that $A_{-1}^+$ is the Seifert matrix of a classical knot (see Burde-Zieschang \cite{bz}, Proposition 8.7). By Fact \ref{fact_livingston_sig}, it suffices to show that $A^+_{-1}$ has vanishing signature function $\widehat{\sigma}_{\omega}$. Observe that $\Delta_{S^+_{-1}}(t)$ has only real roots. Set $a^2=mn$. Then factor to obtain:
\[
\Delta_{S^+_{-1}}(t)=\frac{1}{a^2}(at^2-t-a)(at^2+t-a).
\]
From this it follows that the roots of $\Delta_{S^+_{-1}}(t)$ are $(1\pm \sqrt{1+4mn})/(2\sqrt{mn})$, $(-1\pm \sqrt{1+4mn})/(2\sqrt{mn})$. This implies that the roots of $\Delta_{S^+_{-1}}(t)$ are real and not equal to $\pm 1$. Then the directed signature function $\widehat{\sigma}_{\omega}$ for $A_{-1}^+$ must be constant. In particular, the constant value must be the signature of $B_{-1}=A_{-1}^++(A_{-1}^+)^{\intercal}$. The characteristic polynomial of $B_{-1}$ is $(t^2-2m t-1)(t^2+2nt-1)$. It has roots $-n \pm \sqrt{n^2+1},m \pm \sqrt{m^2+1}$. Hence, $B_{-1}$ has vanishing signature and $A_{-1}^+$ has vanishing signature function. Thus, $A_{-1}^+$ has finite order in $\mathscr{VG}^{\mathbb{Q}}$.
\newline
\newline
Now, $A_0^+$ cannot have order $4$. This follows from Fact \ref{fact_livingston_not_4}, since $\Delta_{K,F}^+(-1)=1$ is not divisible by a prime $p \equiv 3 \mod 4$. By Fact \ref{fact_livingston_order_2}, $A^+_{-1}$ will have order $2$ if there is a symmetric irreducible factor of $\Delta_{S_{-1}^+}(t)$ that has odd exponent in the factorization. We will complete the proof by showing that $\Delta_{S^+_{-1}}(t)$ is itself irreducible over $\mathbb{Q}$ under the stated conditions on $m$ and $n$.  
\newline
\newline
The calculation of the roots of $\Delta_{S^+_{-1}}(t)$ implies that $\Delta_{S^+_{-1}}(t)$ splits in $\mathbb{Q}(\sqrt{mn},\sqrt{4mn+1})$. Note that $\mathbb{Q}(\sqrt{mn}) \cap \mathbb{Q}(\sqrt{4mn+1})=\mathbb{Q}$. Indeed, if $\sqrt{mn}=a+b\sqrt{4mn+1}$ for some $a,b \in \mathbb{Q}$, then $\sqrt{mn}-b \sqrt{4mn+1} \in \mathbb{Q}$. Squaring both sides yields $\sqrt{mn(4mn+1)} \in \mathbb{Q}$. But this cannot be the case since $mn,4mn+1$ are not themselves square and $\gcd(mn,4mn+1)=1$. It follows that $[\mathbb{Q}(\sqrt{mn},\sqrt{4mn+1}):\mathbb{Q}]=4$. Hence, $\Delta_{S^+_{-1}}(t)$ is irreducible over $\mathbb{Q}$ and the proof is complete. \end{proof}

\begin{lemma} \label{lemma_12n} If $n \in \mathbb{N}$ is odd, then $12n+1$ is not a rational square.
\end{lemma}
\begin{proof} If $12n+1=x^2$ for some $x \in \mathbb{N}$, then $12n=(x-1)(x+1)$. Since $x$ is odd, $x-1$ and $x+1$ are even. Since $n$ is odd, $4$ cannot divide either of $x-1$ or $x+1$. This implies that $3n=\left(\tfrac{x-1}{2}\right)\left(\tfrac{x+1}{2}\right)$. Thus, $6$ divides one of $x \pm 1$. Suppose $6a=x-1$. Since $4$ does not divide $x-1$, $a$ is odd. Substituting gives $3n=3a(3a+1)$. But this implies that $n$ is even, since $3a+1$ is even. This is a contradiction. Similarly, $6$ cannot divide $x+1$. Thus, $12n+1$ is not a rational square. 
\end{proof}

\begin{thm} There exist infinitely many knots $K$ in $S^1 \times S^1 \times I$ such that for each $o\in \{1,2,4\}$, $K$ bounds a Seifert surface $F \subset S^1 \times S^1 \times I$ having algebraic concordance order $o$ in $\mathscr{VG}^{\mathbb{Q}}$.
\end{thm}
\begin{proof} Let $m=3$ and choose $n$ to be any odd prime greater than $3$. Then $mn$ is not a rational square. By Lemma \ref{lemma_12n}, $4(3)(n)+1$ is not a rational square. By Lemma \ref{lemma_A_0}, $A_0^+$ has order $4$. By Lemma \ref{lemma_A_minus}, $A_{-1}^+$ has order $2$. Since $A_{-1}^-$ has order $1$, the result follows. \end{proof}

\subsection{Directions for further research} \label{sec_further} Cochran, Orr, and Teichner \cite{COT} showed that there is an infinite filtration of the topological knot concordance group: $\mathscr{C}^{\text{top}} \supset \mathscr{F}_0 \supset \mathscr{F}_{.5} \supset \mathscr{F}_{1} \supset \mathscr{F}_{1.5} \supset \mathscr{F}_2 \supset \cdots$. The group $\mathscr{F}_n$ corresponds to those knots that are $n$-solvable. A knot having vanishing Arf invariant is $0$-solvable. The $.5$-solvable knots are those which are algebraically slice and the $1.5$-solvable knots have trivial Casson-Gordon invariants. Here we have made the first two steps towards extending the COT-filtration to the virtual setting: the Arf invariant and the algebraic concordance group. In future work, we hope to complete the extension to the entire COT-filtration.
\newline
\newline
In \cite{turaev_cobordism}, Turaev used graded matrices to define concordance invariants of knots on surfaces. How are these related to to the virtual algebraic concordance group? While algebraic concordance studies the symmetrized Seifert matrices of a Seifert surface, graded matrices are skew-symmetric matrices over $\mathbb{Z}$. As mentioned in Remark \ref{remark_alex_roots}, skew-symmetric bilinear forms are also needed to classify algebraic concordance of Seifert matrices in general. One would thus expect the Witt groups of Hermitian forms and the theory of $L$-groups to play a fundamental role in a concordance classification of both graded matrices and the coupled algebraic concordance group. We plan to address these technically challenging questions in a subsequent paper.

\subsection*{Acknowledgments} The first author was partially supported by research funds from The Ohio State University, Marion Campus. The second author was supported by the American Mathematical Society and the Simons Foundation through the AMS-Simons Travel Grant. The authors would like to thank N. Petit for helpful conversations about the Arf invariant. For advice and encouragement, we are grateful to S. Carter, D. Freund, P. Pongtanapaisan, and R. Todd.


\bibliographystyle{plain}
\bibliography{acvss_bib}

\end{document}